\documentclass[12pt,reqno]{amsart}
\usepackage[margin=29mm]{geometry}
\usepackage{amsmath}
\usepackage{amsfonts}
\usepackage{amsthm}
\usepackage[utf8]{inputenc}
\usepackage{graphicx}
\usepackage{tikz-cd}
\usepackage{enumitem}
\usepackage{mathtools}
\usepackage{amssymb}
\usepackage{microtype}

\usetikzlibrary{arrows.meta, positioning, shapes.geometric}

\theoremstyle{plain}
\newtheorem{thm}{Theorem}[section]
\newtheorem{lem}[thm]{Lemma}
\newtheorem{prop}[thm]{Proposition}
\newtheorem{cor}[thm]{Corollary}
\newtheorem{ques}[thm]{Question}

\theoremstyle{definition}
\newtheorem{defn}[thm]{Definition}
\newtheorem{rem}[thm]{Remark}
\newtheorem{conj}[thm]{Conjecture}

\DeclareMathOperator{\Spec}{Spec}
\DeclareMathOperator{\Hom}{Hom}
\DeclareMathOperator{\Gal}{Gal}
\DeclareMathOperator{\Ker}{ker}
\DeclareMathOperator{\Coker}{coker}
\DeclareMathOperator{\im}{im}
\DeclareMathOperator{\Aut}{Aut}

\DeclareMathOperator{\gr}{gr}
\DeclareMathOperator{\Frob}{Frob}

\DeclareMathOperator{\Char}{char}

\DeclareMathOperator{\GL}{GL}

\setlist[enumerate,1]{label=\textup{(\roman*)}}

\theoremstyle{plain}

\newtheorem{introthm}{Theorem}

\newtheorem{introcor}[introthm]{Corollary}

\begin{document}

\title{Weights in \'etale cohomology over mixed-characteristic local fields and applications to anabelian geometry}
\author{Yoshiaki Yamamura}
\date{} 

\address{Department of Mathematics, Faculty of Science, Hokkaido University
Kita 10, Nishi 8, Kita-Ku, Sapporo, Hokkaido, 060-0810, Japan}

\email{yamamura.yoshiaki.x5@elms.hokudai.ac.jp}


\begin{abstract}
  In anabelian geometry, Kummer-faithful fields are expected to be suitable as base fields. 
  In recent years, to characterize Kummer-faithfulness in terms of Galois representations, Ozeki and Taguchi defined a notion of (quasi-)highly Kummer-faithful fields as a variant, and posed the following natural question:
  Are sub-$p$-adic fields quasi-highly Kummer-faithful?
  In this paper, by studying the weights of $\ell$-adic \'etale cohomology over mixed-characteristic local fields and discussing the $p$-adic analogue, we give an affirmative answer to this question. 
  Furthermore, we extend the class of base fields from mixed-characteristic local fields to complete discrete valuation fields whose residue fields are algebraic extensions of some finite field, and give equivalent conditions for the vanishing of the coinvariants of $\ell$-adic and of $p$-adic \'etale cohomology.
  As a result, we show that, for mixed-characteristic complete discrete valuation fields whose residue fields are algebraic extensions of some finite field, Kummer-faithfulness and quasi-high Kummer-faithfulness are equivalent.
\end{abstract}

\maketitle
\markboth{YOSHIAKI YAMAMURA}{Weights in \'etale cohomology over MLF's and applications to anabelian geometry}

\setcounter{section}{-1}

\setcounter{tocdepth}{2}

\tableofcontents

\vspace{\baselineskip}
\section*{Introduction}

Anabelian geometry of algebraic curves originated from Grothendieck's conjecture that the geometric and algebraic structures of a hyperbolic algebraic curve can be reconstructed from its \'etale fundamental group.
Initially, it was thought that finitely generated fields over the rational number field would be the suitable base fields for such curves.
However, following the resolution of Grothendieck's conjecture, Mochizuki gave a $p$-adic interpretation, for hyperbolic curves over sub-$p$-adic fields (cf.\ \S 0), of Grothendieck's conjecture (cf.\ \cite{Mochizuki}).
Furthermore, the theory is now being extended to so-called Kummer-faithful fields (cf.\ Definition \ref{Def of KF}), introduced by Mochizuki, for which Kummer theory works effectively in the reconstruction process in anabelian geometry (Note that sub-$p$-adic fields are Kummer-faithful (cf.\ \cite{MochizukiIII})).
In fact, Hoshi showed Grothendieck's conjecture for affine hyperbolic algebraic curves over Kummer-faithful fields (cf.\ \cite{Hoshi}).
Kummer-faithful fields are expected to be suitable base fields for anabelian geometry. 
This raises the question of what kinds of Kummer-faithful fields that are not sub-$p$-adic fields may exist.
Although the resolution of this question is not necessarily straightforward, there are, for example, Ohtani's construction (cf.\ \cite{Ohtani}) and Murotani's result on the Kummer-faithfulness of mixed-characteristic higher local fields (cf.\ \cite{Murotani2}).
In addition, it has been pointed out that Kummer-faithful fields are closely related to the classical finiteness problem for torsion points on abelian varieties (cf.\ \cite{Ozeki}), and are therefore of independent interest in number theory.
In order to characterize Kummer-faithfulness in terms of Galois representations, Ozeki and Taguchi defined a notion of (quasi-)highly Kummer-faithful fields (cf.\ Definition \ref{Def of HKF}) as a variant of Kummer-faithful fields (cf.\ \cite{Ozeki-Taguchi}). 
These fields are closely related to class field theory and are useful in number theory.
For example, by using highly Kummer-faithful fields, one can systematically construct infinite algebraic extensions of the rational number field that are Kummer-faithful.
However, it is known that sub-$p$-adic fields are not highly Kummer-faithful in general. 
This led Ozeki and Taguchi to pose the following question (cf.\ \cite[Question]{Ozeki-Taguchi}):
\begin{center}
  Are sub-$p$-adic fields quasi-highly Kummer-faithful?
\end{center}
Here, we shall say that a perfect field $k$ is quasi-highly Kummer-faithful if any finite extension $K/k$ satisfies the following condition:
\begin{itemize}
  \item[$(\dagger)$] For any proper smooth variety $X$ over $K$, every prime number $\ell$ distinct from characteristic of $K$, every integer $i$ with $0<i\leq 2\dim X$, it holds that 
  $$H^i(X_{\overline{K}},\mathbb{Q} _{\ell})_{G_K}=0$$ 
  --- where $G_K$ is the absolute Galois group of $K$.
\end{itemize}
It was shown by Jannsen that, if $K$ is an MLF (cf.\ \S 0), then $K$ satisfies the condition $(\dagger)$, assuming the monodromy weight conjecture and the $p$-adic monodromy weight conjecture (cf.\ \cite{Jannsen}).
In this paper, we mainly consider the following two problems:
\begin{enumerate}[label=(\Alph*)]
\item Do sub-$p$-adic fields satisfy the condition $(\dagger)$ without assuming the monodromy weight conjecture and the $p$-adic monodromy weight conjecture (i.e., are sub-$p$-adic fields quasi-highly Kummer-faithful)?
\item For what classes of fields do (quasi-)highly Kummer-faithful fields provide a characterization of Kummer-faithful fields in terms of Galois representations?
\end{enumerate}
These problems are natural for the following reasons:
(A) Sub-$p$-adic fields are among the most fundamental base fields in anabelian geometry. 
It is therefore natural to ask whether they are quasi-highly Kummer-faithful, a notion closely related to Kummer-faithful fields, which are expected to be suitable base fields in anabelian geometry.
(B) Although (quasi-)highly Kummer-faithful fields were introduced with the aim of characterizing Kummer-faithful fields in terms of Galois representations, they are expected to form a class that is ``slightly larger" than that of Kummer-faithful fields. 
It is important to determine as large a class of fields as possible for which (quasi-)highly Kummer-faithfulness is equivalent to Kummer-faithfulness. 
This is important not only because the former provides a characterization of the latter in terms of Galois representations, but also because it helps clarify the difference between the two notions.

We consider for (A).
First, we study \'etale cohomology over MLF's and improve Jannsen's calculations (cf.\ Theorems \ref{Jannsen Thm4.2}, \ref{Jannsen Thm5.3}), which were obtained without assuming the monodromy weight conjecture or the $p$-adic monodromy weight conjecture:
\begin{introthm}[Theorems \ref{vanishing over MLF}, \ref{p-adic vanishing over MLF}]
  Let $k$ be an MLF, $G_k$ the absolute Galois group of $k$, $I_k$ the inertia subgroup of $k$, $p$ the characteristic of the residue field of $k$, $\ell$ a prime number distinct from $p$, and $X$ a $d(\geq 1)$-dimensional proper smooth variety.   
  Then the following assertions hold:
  \begin{enumerate}[label=(\roman*)]
  \item The $G_{\underline{k}}$-representation $H^i(X_{\overline{k}},\mathbb{Q} _\ell)^{I_k}$ is mixed with weights in
        \[
        \begin{cases}
          [\max(0, 2i - 2d), \, \max(0, \min(2i - 1, 2d - 1))] & \text{if } i \neq 2d, \\
          \{2d\} & \text{if } i = 2d.
        \end{cases}
        \]
  \item  The $G_{\underline{k}}$-representation $H^i(X_{\overline{k}},\mathbb{Q} _\ell)_{I_k}$ is mixed with weights in
        \[
        \begin{cases}
          \{0\} & \text{if } i=0,\\
          [\min(2d, \max(1, 2i - 2d + 1)), \,  \min(2i, 2d)] & \text{if } i \neq 0 .
        \end{cases}
        \]
    \end{enumerate}
  Moreover, let $\square \in \{\ell, p\}$. 
  Then 
  $$H^i(X_{\overline{k}},\mathbb{Q} _{\square}(r))^{G_k}=0$$
  holds for $i\neq 2d$ and $r\notin \left[\max(0, i - d), \, \max(0, \min(i - 1, d - 1))\right]$. 
  In particular, for $i\neq 0$, we have 
  $H^i(X_{\overline{k}},\mathbb{Q} _{\square})_{G_k}=0$. 
  Thus, any MLF is quasi-highly Kummer-faithful. 
\end{introthm}
From the properties of Galois representations, it is immediate that the vanishing of the invariants and coinvariants of \'etale cohomology over a field $k$ also hold for \'etale cohomology over any subfield of $k$.
On the other hand, we show that the vanishing of the invariants and coinvariants of \'etale cohomology over a field $k$ also propagates to \'etale cohomology over finitely generated extensions of $k$.
In particular, if a perfect field $K$ is finitely generated over a highly Kummer-faithful field (resp.\ a quasi-highly Kummer-faithful field), then $K$ is highly Kummer-faithful (resp.\ quasi-highly Kummer-faithful).
As a result, we give an affirmative answer to the question posed by Ozeki and Taguchi:
\begin{introthm}[Proposition \ref{f.g. over HKF}, Theorem \ref{QHKFness of sub-p-adic fields}]
  Let $k$ be a field. 
  If any finite extension $K/k$ satisfies the condition $(\dagger)$, then any finitely generated extension of $k$ also satisfies the condition $(\dagger)$. 
  Thus, any sub-$p$-adic field is quasi-highly Kummer-faithful. 
\end{introthm}

We consider for (B).
Murotani gave equivalent conditions for an algebraic extension of an FF (cf.\ \S 0) to be Kummer-faithful or torally Kummer-faithful (cf.\ Definition \ref{Def of KF}) (cf.\ \cite{Murotani1}).
We first consider a highly Kummer-faithful analogue of this result and obtain the following theorem:
\begin{introthm}[Proposition \ref{vanishing of coinvariants over alg ext over FF}, Theorem \ref{HKFness over alg ext over FF}]
  For a prime number $p$, let $\mathbb{F} /\mathbb{F} _p$ be an algebraic extension, $G_{\mathbb{F} }$ the absolute Galois group of $\mathbb{F} $, and $\ell $ a prime number distinct from $p$. 
  Then the following conditions are equivalent:
  \begin{enumerate}[label=(\roman*)]
    \item The maximal pro-$\ell$ quotient of $G_{\mathbb{F} }$ is nontrivial.
    \item For any proper smooth variety $X$ over $\mathbb{F} $ and every pair of integers $i,r$ with $i\neq 2r$, it holds that $H^i(X_{\overline{\mathbb{F} }},\mathbb{Q} _{\ell}(r))^{G_{\mathbb{F} }}=0.$ 
    \item For any proper smooth variety $X$ over $\mathbb{F} $ and every integer $i\neq 0$, it holds that $H^i(X_{\overline{\mathbb{F} }},\mathbb{Q} _{\ell})_{G_{\mathbb{F} }}=0$. 
  \end{enumerate}
  Moreover, the following conditions are equivalent:
  \begin{enumerate}[label=(\arabic*)]
    \item For every prime number $\ell$ distinct from $p$, the maximal pro-$\ell$ quotient of $G_{\mathbb{F} }$ is nontrivial. 
    \item $\mathbb{F}$ is highly Kummer-faithful. 
    \item $\mathbb{F}$ is quasi-highly Kummer-faithful. 
    \item $\mathbb{F}$ is torally Kummer-faithful. 
  \end{enumerate}
\end{introthm}
Furthermore, Murotani showed that, for an MCDVF (cf.\ \S 0), if its residue field is pre-Kummer-faithful (cf.\ Definition \ref{Def of KF}), then the MCDVF itself is (also) Kummer-faithful (cf.\ \cite{Murotani2}).
Hence, as a quasi-highly Kummer-faithful analogue of this fact, we consider MCDVF's whose residue fields are highly Kummer-faithful.
By extending Theorem A to a quasi-finite-extension-like setting, we obtain the following theorem:
\begin{introthm}[Propositions \ref{l-adic vanishing over CDVF}, \ref{p-adic vanishing over MCDVF}, Theorem \ref{QHKFness over MCDVF}]
  Let $k$ be a CDVF, $\underline{k}$ the residue field of $k$, $G_{k}$ the absolute Galois group of $k$, $G_{\underline{k}}$ the absolute Galois group of $\underline{k}$, $\ell$ a prime number distinct from a prime number $p$, and $\square \in \{\ell, p\}$.  
  Assume that $\underline{k}/\mathbb{F} _{p}$ is an algebraic extension. 
  Consider the following conditions:
  \begin{enumerate}[label=(\roman*)]
    \item The maximal pro-$\square$ quotient of $G_{\underline{k}}$ is nontrivial.
    \item For any $d$-dimensional proper smooth variety $X$ over $k$, every integer $i\neq 2d$, and every integer $r\notin \left[\max(0, i - d), \, \max(0, \min(i - 1, d - 1))\right]$, it holds that 
          $H^i(X_{\overline{k}},\mathbb{Q} _{\square}(r))^{G_k}=0$. 
    \item For any proper smooth variety $X$ over $k$ and every integer $i\neq 0$, it holds that $H^i(X_{\overline{k}},\mathbb{Q} _{\square})_{G_k}=0$. 
    \item For any abelian variety $A$ over $k$, it holds that $V_{\square}(A)^{G_k}=0$. 
    \item For any semi-abelian variety $B$ over $k$, it holds that $V_{\square}(B)^{G_k}=0$. 
    \item $k$ is stably $\mu _{\square ^{\infty}}$-finite (cf.\ Definition \ref{Def of stably mu-finite}). 
  \end{enumerate}
  Then,
  \begin{itemize}
    \item $(i)\Leftrightarrow (ii)\Leftrightarrow (iii) \Leftrightarrow (iv)\Leftrightarrow (v)\Leftrightarrow (vi)$ holds if $\square = \ell$, 
    \item $(i)\Leftrightarrow (ii)\Leftrightarrow (iii)\Leftrightarrow (v)$ holds if $k$ is an MCDVF and $\square =p$.
  \end{itemize}
  Moreover, if $k$ is an MCDVF, then the following conditions are equivalent:
  \begin{enumerate}[label=(\arabic*)]
    \item $\underline{k}$ is quasi-finite. 
    \item $k$ is quasi-highly Kummer-faithful. 
    \item $k$ is Kummer-faithful. 
    \item $k$ is $\mathfrak{Primes}^{\infty}$-semi-AV-tor-finite (cf.\ Definition \ref{Def of semi-AV-tor-finite}). 
  \end{enumerate}
\end{introthm}
Ozeki gave a criterion for Kummer-faithfulness of infinite algebraic extensions of MLF's (cf.\ \cite{Ozeki}).
Theorem D gives a criterion for quasi-highly Kummer-faithful fields in situations like the completion of an infinite algebraic extension of an MLF.
Theorem D immediately implies a criterion for quasi-highly Kummer-faithfulness of infinite unramified extensions of MLF's:
\begin{introcor}[Corollary \ref{QHKFness of unram etx over MLF}]
  Let $k$ be an MLF, $K/k$ an unramified extension. 
  Then the following conditions are equivalent:
  \begin{enumerate}[label=(\roman*)]
    \item $K/k$ is a quasi-finite extension. 
    \item $K$ is quasi-highly Kummer-faithful. 
    \item $K$ is Kummer-faithful. 
    \item $K$ is $\mathfrak{Primes}^{\infty}$-semi-AV-tor-finite.  
  \end{enumerate}
\end{introcor}

\vspace{\baselineskip}
\section*{Acknowledgments}
  I would like to express my deep gratitude to Seidai Yasuda, my supervisor, for his invaluable guidance.
  I would also like to express my deep gratitude to Yoshiyasu Ozeki for his helpful advice.

  \vspace{\baselineskip}
\section{Notations and conventions}
\noindent
\textbf{Numbers:}

We shall write
\begin{itemize}
  \item $\mathbb{Z} $ for the set of integers, 
  \item $\mathbb{Q} $ for the set of rational numbers, 
  \item $\mathbb{R} $ for the set of real numbers, 
  \item $\mathbb{C} $ for the set of complex numbers,
  \item $\mathfrak{Primes}$ for the set of prime numbers.
\end{itemize}
We shall write $\mathbb{Z}_{\geq 0}$ (resp.\ $\mathbb{Z}_{>0}$) for the set of nonnegative integers (resp.\ positive integers). 
For $a,b\in\mathbb{R}$, we shall write $[a,b]$ for 
$\{x\in \mathbb{R} \,|\, a\leq x\leq b\}.$\\

\noindent
\textbf{Fields:}

Let $k$ be a field.
Then we shall write
\begin{itemize}
  \item $k(t)$ for the $1$-variable rational function field over $k$, 
  \item $k(\!(t)\!)$ for the $t$-adic completion of $k(t)$, 
  \item $k^{\text{perf}}$ for the perfect closure of $k$, 
  \item $k^{\text{sep}}$ for the separable closure of $k$,
  \item $\overline{k}$ for the algebraic closure of $k$.
\end{itemize}
For $p\in \mathfrak{Primes}$ and $n\in \mathbb{Z} _{> 0}$, we shall write 
\begin{itemize}
  \item $\mathbb{Z} _p$ for the $p$-adic completion of $\mathbb{Z} $, 
  \item $\mathbb{Q} _p$ for the fraction field of $\mathbb{Z} _p$, 
  \item $\mathbb{C} _p$ for the $p$-adic completion of $\overline{\mathbb{Q} }_p$,
  \item $\mathbb{F} _{p^n}$ for the finite field of cardinality $p^n$.
\end{itemize} 
Let $k$ be a field.
We shall say that $k$ is an FF if $k$ is a finite field (where ``FF" is
understood as an abbreviation for ``Finite Field"). 
We shall say that $k$ is a CDVF (resp.\ a PCDVF, resp.\ an MCDVF) if $k$ is a complete discrete valuation field (resp.\ a positive-characteristic complete discrete valuation field, resp.\ an mixed-characteristic complete discrete valuation field) (where ``CDVF" (resp.\ ``PCDVF", resp.\ ``MCDVF") is
understood as an abbreviation for ``Complete Discrete Valuation Field" (resp.\ ``Positive-characteristic Complete Discrete Valuation Field", resp.\ ``Mixed-characteristic Complete Discrete Valuation Field")). 
We shall say that $k$ is a PLF (resp.\ an MLF) if $k$ is isomorphic to a finite extension of  $\mathbb{F} _p(\!(t)\!)$ (resp.\ $\mathbb{Q} _p$) for some $p\in \mathfrak{Primes}$ (where ``PLF" (resp.\ ``MLF") is understood as an abbreviation for ``Positive-characteristic Local Field" (resp.\ ``Mixed-characteristic Local Field")). 
We shall say that $k$ is a sub-$p$-adic field if $k$ is isomorphic to a subfield of a finitely generated field over $\mathbb{Q} _p$.
\\

\noindent
\textbf{Modules:}

Let $M$ be a $\mathbb{Z} $-module and let $p\in \mathfrak{Primes}$. 
Then we shall write $M_{\text{div}}$ (resp.\ $M_{p\text{-div}}$) for the set of divisible elements (resp.\ $p$-divisible elements) of $M$, i.e., 
\[
M_{\text{div}}=\bigcap _{n\in \mathbb{Z} _{>0}} nM, \ M_{p\text{-div}}=\bigcap _{n\in \mathbb{Z} _{>0}} p^nM .
\] 
For $n\in \mathbb{Z} _{>0}$, we shall write
\begin{itemize}
  \item $M[n]$ for the kernel of multiplication by $n$,
  \item $M[p^{\infty}]\overset{\mathrm{def}}{=}\bigcup _{n\in \mathbb{Z} _{>0}}M[p^n]$. 
\end{itemize}
Furthermore, assume that a group $G$ acts on $M$. 
Then we shall write $M^G$ (resp.\ $M_G$) for the invariants (resp.\ coinvariants) of $M$ under the action of $G$. 
For $r\in \mathbb{Z} $, we shall write $\mathbb{Q} _p(r)$ for the $r$-th Tate twist of $\mathbb{Q} _p$.
For $p \in \mathfrak{Primes}$ and semi-abelian variety $A$, we shall write $V_{p}(A)$ for the $p$-adic rational Tate module of $A$. 
\\

\noindent
\textbf{Profinite groups:}

Let $G$ be a profinite group and let $p\in \mathfrak{Primes}$. 
Then we shall write $G^p$ for the maximal pro-$p$ quotient of $G$. 
Let $k$ be a field. 
Then we shall write $G_k\overset{\mathrm{def}}{=}\Gal (k^{\text{sep}}/k)$ for the absolute Galois group of $k$. 
\\

\noindent
\textbf{Period rings:}

Let $\mathcal{O} _{\mathbb{C} _p}$ be the ring of integers of $\mathbb{C} _p$. 
Then we shall write $\mathcal{O} _{\mathbb{C} _p}^\flat $ for the inverse limit of the inverse system
\[{\mathcal{O} _{\mathbb{C} _p}}/p\xlongleftarrow{\Frob _p} {\mathcal{O} _{\mathbb{C} _p}}/p\xlongleftarrow{\Frob _p} {\mathcal{O} _{\mathbb{C} _p}}/p\xlongleftarrow{\Frob _p} \dots\]
--- where $\Frob _p$ is the $p$-th power Frobenius map. 
We shall write $A_{\text{inf}}\overset{\mathrm{def}}{=} W(\mathcal{O} _{\mathbb{C} _p}^\flat)$ for the Witt ring of $\mathcal{O} _{\mathbb{C} _p}^\flat $. 
For $\square \in \{\text{dR}, \text{crys}, \text{st}\}$, we shall write $B_{\square}$ for the $p$-adic period ring. 
\\

\noindent
\textbf{Schemes:}

Let $k$ be a field, $X$ a scheme over a ring $A$, $A\to k$ a homomorphism.
Then we shall write $X_{k}\overset{\mathrm{def}}{=}X\times _{\Spec A}\Spec k$ for the base change of $X$ to $k$.
We shall say that $X$ is a variety if $X$ is a separated geometrically integral scheme of finite type over a field. 
We shall say that $X$ is a SNCL-variety if $X$ is a simple normal crossing log variety (where ``SNCL" is understood as an abbreviation for ``Simple Normal Crossing Log"). 
Let $k$ be a field. 
Then we shall write $\mathbb{P} _k ^d$ (resp.\ $\mathbb{G} _m$) for the $d$-dimensional projective space over $k$ (resp.\ the multiplicative group scheme).  
\\

\noindent
\textbf{Cohomologies:}

Let $k$ be a field. 
Then, for any $d$-dimensional proper smooth variety $X$ over $k$, each $\ell \in \mathfrak{Primes}$ distinct from the characteristic of $k$, and each integer $i$ satisfying $0\leq i\leq 2d$, we shall write $H^i(X_{\overline{k}},\mathbb{Q} _{\ell})$ for the $i$-th $\ell$-adic \'etale cohomology. 
Let $\kappa$ be a field of characteristic $p>0$ and $K_0(\kappa)\overset{\mathrm{def}}{=}W(\kappa)[1/p]$ the fraction field of the Witt ring of $\kappa$. 
Then, for a $d$-dimensional proper smooth variety (resp.\ a $d$-dimensional proper SNCL-variety) $X$ over $\kappa$ and an integer $i$ satisfies $0\leq i\leq 2d$, we shall write $H_{\text{crys}} ^i (X/K_0(\kappa))$ (resp.\ $H_{\text{log-crys}} ^i (X/K_0(\kappa))$) for the $i$-th crystalline cohomology (resp.\ the $i$-th log crystalline cohomology). 

\vspace{\baselineskip}
\section{Weights in \'etale cohomology over MLF's}

In this section, we improve the Jannsen's calculation (cf.\ Theorem \ref{Jannsen Thm4.2}) on the weights of the inertia-invariants of $\ell$-adic \'etale cohomology over MLF's. 
As a result, we show the vanishing of the coinvariants of the $\ell$-adic \'etale cohomology under the action of the absolute Galois group.
Furthermore, we discuss the $p$-adic analogue and improve Jannsen's calculation (cf.\ Theorem \ref{Jannsen Thm5.3}).\\

In the remainder of this section, let $k$ be a CDVF.
Then we shall write
\begin{itemize}
    \item $\mathcal{O} _k$ for the ring of integers,
    \item $\pi _k$ for a uniformizer of $\mathcal{O} _k$,
    \item $\underline{k}$ for the residue field of $\mathcal{O} _k$,
    \item $\overline{\underline{k}}$ for the algebraic closure of $\underline{k}$,
    \item $p$ for the characteristic of $\underline{k}$,
    \item $\ell $ for a prime number distinct from $p$,
    \item $I_k$ for the inertia subgroup of $k$,
    \item $t_\ell: I_k \to  \mathbb{Z} _\ell(1)$, $\sigma \mapsto \left(\sigma \left(\pi_k ^{1/{\ell ^m}}\right)\middle/ \pi_k ^{1/{\ell ^m}}\right)_m$ for the canonical surjection,
    \item $\sigma _0$ for an element of $I_k$ such that $t_{\ell}(\sigma _0)$ topologically generates $\mathbb{Z} _{\ell}(1)$,
    \item $X$ for a proper smooth variety over $k$,
    \item $d$ for the dimension of $X$, 
    \item $\mathfrak{X} \to \Spec \mathcal{O} _k$ for a proper model of $X$,
    \item $i$ for an integer satisfying $0\leq i\leq 2d$,
    \item $r$ for an integer.
\end{itemize}
Furthermore, if $X$ has semi-stable reduction, then, for each integer $j\geq  0$, we shall write 
$$\mathfrak{X} _{\underline{k}}^{(j)}$$ 
for the disjoint union of all $(j+1)$-fold intersections of the distinct irreducible components of $\mathfrak{X} _{\underline{k}} $.  \\

Grothendieck's monodromy theorem implies that, after replacing $k$ by a finite separable extension, we may assume that the action of $I_k$ on $H^i(X_{\overline{k}},\mathbb{Q}_{\ell})$ is unipotent.
Define 
\[N\overset{\mathrm{def}}{=} \sum _{m=1} ^{\infty} (-1)^{m-1}\frac{(\sigma _0 -1)^m}{m}.\]
Then $N$ is the monodromy operator on $H^i(X_{\overline{k}},\mathbb{Q}_{\ell})$.
That is, there is a unique increasing filtration $M_{\bullet}$ on $H^i(X_{\overline{k}},\mathbb{Q}_{\ell})$, called the monodromy filtration, characterized by the following properties:
\begin{enumerate}
\item For every sufficiently large integer $j$, it holds that $M_j=H^i(X_{\overline{k}},\mathbb{Q}_{\ell})$ and $M_{-j-1}=0$.
\item For every integer $j$, it holds that $NM_j\subseteq M_{j-2}$.

\item For every nonnegative integer $j$, the induced map $N^j:\gr^M_jH^i(X_{\overline{k}},\mathbb{Q}_{\ell}) \to \gr^M_{-j}H^i(X_{\overline{k}},\mathbb{Q}_{\ell})$
is an isomorphism.
\end{enumerate}
Furthermore, if $k$ is a PLF or an MLF, then there is the following conjecture:
\begin{conj}[monodromy weight conjecture]\label{monodromy weight conjecture}
  The $G_{\underline{k}}$-representation $\gr ^M _jH^i(X_{\overline{k}},\mathbb{Q}_{\ell})$ is pure of weight $i+j$.
\end{conj}

\begin{rem}\label{remark of monodromy weight conjecture}
  This conjecture was solved by Ito in the case where $k$ is a PLF (cf.\ \cite{Ito}). 
  In the case where $k$ is an MLF, it was solved by Rapoport and Zink under the assumption that $d \leq 2$ (cf.\ \cite{Rapoport-Zink}).
\end{rem}

\begin{thm}[{\cite[Theorem 4.2]{Jannsen}}]\label{Jannsen Thm4.2}
  Let $k$ be a PLF or an MLF.
  Then the following assertions hold:
  \begin{enumerate}[label=(\roman*)]
    \item $H^i(X_{\overline{k}},\mathbb{Q} _\ell)^{I_{k}}$ is a mixed $G_{\underline{k}}$-representation with weights in $[\max (0, 2i-2d),\, \min (2i, 2d+2)]$.
    \item $H^i(X_{\overline{k}},\mathbb{Q} _\ell)_{I_{k}}$ is a mixed $G_{\underline{k}}$-representation with weights in $[\max (-2, 2i-2d),\, \min (2i, 2d)]$.
  \end{enumerate}
  Thus, for any $r\notin [\max (0, i-d), \min (i, d+1)]$, we have
  $$H^i(X_{\overline{k}},\mathbb{Q} _\ell(r))^{G_{k}}=0.$$ 
\end{thm}

\begin{cor}[{\cite[Corollary 4.3]{Jannsen}}]\label{Jannsen Cor4.3}
  Let $k$ be a PLF or an MLF.
  Assume that Conjecture \ref{monodromy weight conjecture} holds for $X$.
  Then the following assertions hold:
  \begin{enumerate}[label=(\roman*)]
    \item $H^i(X_{\overline{k}},\mathbb{Q} _\ell)^{I_{k}}$ is a mixed $G_{\underline{k}}$-representation with weights in $[\max (0, 2i-2d),\, i]$.
    \item $H^i(X_{\overline{k}},\mathbb{Q} _\ell)_{I_{k}}$ is a mixed $G_{\underline{k}}$-representation with weights in $[i,\, \min (2i, 2d)]$.
  \end{enumerate}
  Thus, for any $r\notin\left[\max (0, i-d), \, \frac{i}{2}\right]$, we have
  $$H^i(X_{\overline{k}},\mathbb{Q} _\ell(r))^{G_{k}}=0.$$ 
\end{cor}

Assume that $X$ has strictly semi-stable reduction.
Then there is the following $\ell$-adic weight spectral sequence (cf.\ \cite[Satz 2.10]{Rapoport-Zink}, \cite[Corollary 2.2.4]{Saito}):
 \[
  E_1^{a,b}
  = \bigoplus_{r \ge \max(0, -a)}
  H^{b-2r} ( \mathfrak{X} _{\overline{\underline{k}}}^{(a+2r)}, \mathbb{Q}_\ell(-r) )
  \Longrightarrow
  H^{a+b}( X_{\overline{k}}, \mathbb{Q}_\ell ) .
 \]
 If $k$ is a PLF or an MLF, then it follows immediately from the finiteness of the residue field that the following properties hold:
\begin{enumerate}
  \item The $G_{\underline{k}}$-representation $E_1^{a,b}$ is pure of weight $b$. 
  Thus, the $G_{\underline{k}}$-representation $E_2^{a,b}$ is also pure of weight $b$. 
  \item The filtration on $H^i(X_{\overline{k}},\mathbb{Q} _\ell)$ induced by the convergence of the $\ell$-adic weight spectral sequence gives the weight filtration on $H^i(X_{\overline{k}},\mathbb{Q} _\ell)$. 
  \item The $\ell$-adic weight spectral sequence degenerates at the $E_2$-page.
\end{enumerate}
Even if $k$ is a CDVF,  the $\ell$-adic weight spectral sequence degenerates at the $E_2$-page by \cite[Theorem 0.1]{Nakayama}.
Furthermore, there is a morphism between the following $\ell$-adic weight spectral sequences: 
\[
\begin{tikzcd}
E_1^{a,b}
=
\displaystyle
\bigoplus_{r \ge \max(0,-a)}
H^{\,b-2r}
(
\mathfrak{X}_{\overline{\underline{k}}}^{(a+2r)},
\mathbb{Q}_\ell(-r)
)
\arrow[r, Rightarrow]
\arrow[d, "\mathrm{id}\otimes t_\ell(\sigma _0)"]
&
H^{a+b}
(
X_{\overline{k}},
\mathbb{Q}_\ell
)
\arrow[d, "N"]
\\
E_1^{a+2,b-2}
=
\displaystyle
\bigoplus_{r \ge \max(0,-a)}
H^{\,b-2r}
(
\mathfrak{X}_{\overline{\underline{k}}}^{(a+2r)},
\mathbb{Q}_\ell(-r+1)
)
\arrow[r, Rightarrow]
&
H^{a+b}
(
X_{\overline{k}},
\mathbb{Q}_\ell
).
\end{tikzcd}
\]

\begin{prop}\label{isom of N^d}
   Assume that $d\geq 1$ and $X$ has strictly semi-stable reduction. 
   Then $N^d:E_2^{-d,2d}\to E_2^{d,0}$ is an isomorphism. 
\end{prop}
\begin{proof}
  For simplicity, we omit Tate twists and coefficients. 
  It suffices to prove that the following natural map is an isomorphism:
  \[
  \Ker \left(H^0(\mathfrak{X} _{\overline{\underline{k}}}^{(d)})\xrightarrow{\mathrm{Gys}}H^2(\mathfrak{X} _{\overline{\underline{k}}}^{(d-1)})\right)\longrightarrow \Coker \left(H^0(\mathfrak{X} _{\overline{\underline{k}}}^{(d-1)})\xrightarrow{\mathrm{Res}}H^0(\mathfrak{X} _{\overline{\underline{k}}}^{(d)})\right)
  \]
  --- where $\mathrm{Gys}$ (resp.\ $\mathrm{Res}$) is the sum of the Gysin maps (resp.\ restriction maps), each multiplied by an appropriate sign $(\pm 1)$. 
  Since $\mathrm{Gys}$ and $\mathrm{Res}$ are dual to each other, it is enough to prove that the above map is injective.  
  Therefore, it suffices to prove that the nondegenerate bilinear pairing induced by the cup product
  $
  B: H^0(\mathfrak{X} _{\overline{\underline{k}}}^{(d)}) \times H^0(\mathfrak{X} _{\overline{\underline{k}}}^{(d)}) \to \mathbb{Q} _{\ell}
  $
  remains nondegenerate when restricted to $\im \mathrm{Res}\times \im \mathrm{Res}$. 
  Let $H^0(\mathfrak{X} _{\overline{\underline{k}}}^{(d)})_\mathbb{Q} $ (resp.\ $H^0(\mathfrak{X} _{\overline{\underline{k}}}^{(d-1)})_\mathbb{Q} $) be the space of $\mathbb{Q} $-valued locally constant functions on $\mathfrak{X} _{\overline{\underline{k}}} ^{(d)}$ (resp.\ $\mathfrak{X} _{\overline{\underline{k}}} ^{(d-1)}$).
  Let $\mathrm{Res}_\mathbb{Q}:H^0(\mathfrak{X} _{\overline{\underline{k}}}^{(d-1)})_\mathbb{Q} \to H^0(\mathfrak{X} _{\overline{\underline{k}}}^{(d)})_\mathbb{Q} $ be the natural morphism and  
  $B_\mathbb{Q}:H^0(\mathfrak{X} _{\overline{\underline{k}}}^{(d)})_\mathbb{Q} \times H^0(\mathfrak{X} _{\overline{\underline{k}}}^{(d)})_\mathbb{Q} \to \mathbb{Q} $ the standard inner product. 
  Then each of $\square \in \{H^0(\mathfrak{X} _{\overline{\underline{k}}}^{(d)}), H^0(\mathfrak{X} _{\overline{\underline{k}}}^{(d-1)}), \mathrm{Res}, B \}$ admits the corresponding combinatorial $\mathbb{Q} $-structure $\square_{\mathbb{Q}} $. 
  Since the restriction of $B_\mathbb{Q}$ to $\im \mathrm{Res} _{\mathbb{Q} }\times \im \mathrm{Res} _{\mathbb{Q} }$ is nondegenerate, the restriction of $B$ to $\im \mathrm{Res}\times \im \mathrm{Res}$ is also nondegenerate. 
\end{proof}

\begin{thm}\label{vanishing over MLF}
   Let $k$ be an MLF.
   Assume that $d\geq 1$. 
   Then the following  assertions hold:
  \begin{enumerate}[label=(\roman*)]
  \item The $G_{\underline{k}}$-representation $H^i(X_{\overline{k}},\mathbb{Q} _\ell)^{I_k}$ is mixed with weights in
        \[
        \begin{cases}
          [\max(0, 2i - 2d), \, \max(0, \min(2i - 1, 2d - 1))] & \text{if } i \neq 2d, \\
          \{2d\} & \text{if } i = 2d.
        \end{cases}
        \] 
  \item The $G_{\underline{k}}$-representation $H^i(X_{\overline{k}},\mathbb{Q} _\ell)_{I_k}$ is mixed with weights in
        \[
        \begin{cases}
          \{0\} & \text{if } i=0,\\
          [\min(2d, \max(1, 2i - 2d + 1)), \,  \min(2i, 2d)] & \text{if } i \neq 0.
        \end{cases}
        \]
    \end{enumerate}
  Thus, for $i\neq 2d$ and $r\notin \left[\max(0, i - d), \, \max(0, \min(i - 1, d - 1))\right]$, we have 
  $$H^i(X_{\overline{k}},\mathbb{Q} _\ell(r))^{G_k}=0.$$
  In particular, for $i\neq 0$, we have 
  $H^i(X_{\overline{k}},\mathbb{Q} _\ell)_{G_k}=0.$ 
\end{thm}
\begin{proof}
  By Poincar\'e duality, $(ii)$ follows from $(i)$. 
  Therefore, we prove $(i)$. 
  By de Jong's alteration (cf.\ \cite[Theorem 6.5]{deJong}), we may assume that $X$ has projective strictly semi-stable reduction.
  Then we consider the following $\ell$-adic weight spectral sequence:
  \[
  E_1^{a,b}
  = \bigoplus_{r \ge \max(0, -a)}
  H^{b-2r} ( \mathfrak{X} _{\overline{\underline{k}}}^{(a+2r)}, \mathbb{Q}_\ell(-r) )
  \Longrightarrow
  H^{a+b}( X_{\overline{k}}, \mathbb{Q}_\ell ). 
 \]
 By Propositions \ref{isom of N^d}, it holds that 
 $
    H^d(X_{\overline{k}},\mathbb{Q} _\ell)^{I_k} 
    \subseteq  E_2^{-d+1,2d-1}\oplus E_2^{-d+2,2d-2}\oplus \cdots \oplus E_2^{d,0}.                                                 
 $ 
   Since the $G_{\underline{k}}$-representation $E_2 ^{a,b}$ is pure of weight $b$, the $G_{\underline{k}}$-representation $H^d(X_{\overline{k}},\mathbb{Q} _\ell)^{I_k}$ is mixed with weights in $[0,\, 2d-1]$. 

  In the remainder of the proof, we use the induction on $d$. 
  The assertion holds for the case that $d=1$ from the above argument. 
  Assume that the assertion holds for $d-1$. 
  Let $Z$ be a projective smooth hyperplane section of $X$ over $k$. 
  Then, by the weak Lefschetz theorem, for every $i\leq d-1$, the morphism $H^i(X_{\overline{k}},\mathbb{Q} _{\ell})\to H^i(Z_{\overline{k}},\mathbb{Q} _{\ell})$ induced by the inclusion is injective. 
  Thus, by the induction hypothesis, for every $i\leq d-1$, the $G_{\underline{k}}$-representation $H^i(X_{\overline{k}},\mathbb{Q} _\ell)^{I_k}$ is mixed with weights in $[0,\, \max(0,2i-1)]$. 
  On the other hand, by the hard Lefschetz theorem, for every $i\geq d+1$, there is a $G_k$-isomorphism $H^i(X_{\overline{k}},\mathbb{Q} _\ell)\cong H^{2d-i}(X_{\overline{k}},\mathbb{Q} _\ell(i-d))$. 
  Thus, for every $d+1\leq i< 2d $, the $G_{\underline{k}}$-representation $H^i(X_{\overline{k}},\mathbb{Q} _\ell)^{I_k}$ is mixed with weights in $[2i-2d,\ 2d-1]$. 
  Therefore, for every $i\neq 2d$, the $G_{\underline{k}}$-representation $H^i(X_{\overline{k}},\mathbb{Q} _\ell)^{I_k}$ is mixed with weights in $[\max(0, 2i - 2d), \, \max(0, \min(2i - 1, 2d - 1))]$. 
\end{proof}

\begin{prop}\label{isom of N^d-1}
   Let $k$ be an MLF. 
   Assume that $d\geq 2$ and $X$ has strictly semi-stable reduction.
   Then $N^{d-1}:E_2^{-d+1,2d-1}\to E_2^{d-1,1}$ is an isomorphism. 
\end{prop}
\begin{proof}
  For simplicity, we omit Tate twists and coefficients. 
  It suffices to prove that the following natural map is an isomorphism:
  \[
  \Ker \left(H^1(\mathfrak{X} _{\overline{\underline{k}}}^{(d-1)})\xrightarrow{\mathrm{Gys}}H^3(\mathfrak{X} _{\overline{\underline{k}}}^{(d-2)})\right)\longrightarrow \Coker \left(H^1(\mathfrak{X} _{\overline{\underline{k}}}^{(d-2)})\xrightarrow{\mathrm{Res}}H^1(\mathfrak{X} _{\overline{\underline{k}}}^{(d-1)})\right)
  \]
  --- where $\mathrm{Gys}$ (resp.\ $\mathrm{Res}$) is the sum of the Gysin maps (resp.\ restriction maps), each multiplied by an appropriate sign $(\pm 1)$.  
  Thus, the assertion follows from \cite[Satz 2.13]{Rapoport-Zink}. 
\end{proof}

\begin{prop}\label{more detail wt over MLF}
  Let $k$ be an MLF. 
  Assume that $d\geq 2$.  
  Then, for every $2\leq i\leq 2d-2$, the following assertions hold:
  \begin{enumerate}[label=(\roman*)]
    \item  The $G_{\underline{k}}$-representation $H^i(X_{\overline{k}},\mathbb{Q} _\ell)^{I_k}$ is mixed with weights in $[\max(0, 2i - 2d), \, \min(2i - 2, 2d - 2)] $.
    \item  The $G_{\underline{k}}$-representation $H^i(X_{\overline{k}},\mathbb{Q} _\ell)_{I_k}$ is mixed with weights in $[\max(2, 2i - 2d + 2), \,  \min(2i, 2d)] $.
  \end{enumerate}
\end{prop}
\begin{proof}
  By Poincar\'e duality, $(ii)$ follows from $(i)$. 
  Therefore, we prove $(i)$.  
  We may assume that $X$ has projective strictly semi-stable reduction.
  Then we consider the following $\ell$-adic weight spectral sequence:
\[
  E_1^{a,b}
  = \bigoplus_{r \ge \max(0, -a)}
  H^{b-2r} ( \mathfrak{X} _{\overline{\underline{k}}}^{(a+2r)}, \mathbb{Q}_\ell(-r) )
  \Longrightarrow
  H^{a+b}( X_{\overline{k}}, \mathbb{Q}_\ell ) .
 \]
 By Propositions \ref{isom of N^d-1}, it holds that 
 $
    H^d(X_{\overline{k}},\mathbb{Q} _\ell)^{I_k} 
    \subseteq  E_2^{-d+2,2d-2}\oplus E_2^{-d+3,2d-3}\oplus \cdots \oplus E_2^{d,0} .                                                
 $  

 Since the $G_{\underline{k}}$-representation $E_2 ^{a,b}$ is pure of weight $b$, the $G_{\underline{k}}$-representation $H^d(X_{\overline{k}},\mathbb{Q} _\ell)^{I_k}$ is mixed with weights in $[0,\, 2d-2]$. 

In the remainder of the proof, we use the induction on $d$. 
The assertion holds for the case that $d=2$ from the above argument. 
Assume that the assertion holds for $d-1$. 
As in the latter part of the proof of Theorem \ref{vanishing over MLF}, one can prove that the $G_{\underline{k}}$-representation $H^i(X_{\overline{k}},\mathbb{Q} _\ell)^{I_k}$ is mixed with weights in $[0,\, \max(0,2i-2)]$ for every $i\leq d-1$ and in $[2i-2d,\ 2d-2]$ for every $d+1\leq i< 2d $. 
Thus, for every $2\leq i \leq 2d-2$, the $G_{\underline{k}}$-representation $H^i(X_{\overline{k}},\mathbb{Q} _\ell)^{I_k}$ is mixed with weights in $[\max(0, 2i - 2d), \, \max(0, \min(2i - 2, 2d - 2))]$. 
\end{proof}

For later use, we state the following lemma more general version of the argument used in the proofs of Theorem \ref{vanishing over MLF} and Proposition \ref{more detail wt over MLF}:

\begin{lem}\label{key lem of vanishing}
  Let $\square \in \{\ell, p\}$. 
  Assume that for any $n$-dimensional projective smooth variety $Y$ over $k$ and every $r\notin [0,\, n-1]$, it holds that $H^n(Y_{\overline{k}},\mathbb{Q} _{\square}(r))^{G_k}=0$. 
  Then, for every $i\neq 2d$ and $r\notin \left[\max(0, i - d), \, \max(0, \min(i - 1, d - 1))\right]$, we have 
    $$H^i(X_{\overline{k}},\mathbb{Q} _{\square}(r))^{G_k}=0.$$
\end{lem}
\begin{proof}
  We may assume that $X$ has projective strictly semi-stable reduction. 
  We use the induction on $d$.  
  Assume that $d\geq 2$. 
  Let $Z$ be a projective smooth hyperplane section of $X$ over $k$.
  Then, by the weak Lefschetz theorem, for every $i\leq d-1$, the morphism $H^i(X_{\overline{k}},\mathbb{Q} _{\square})\to H^i(Z_{\overline{k}},\mathbb{Q} _{\square})$ induced by the inclusion is injective. 
  Thus, by the induction hypothesis, for every $i\leq d-1$ and $r\notin [0,\, \max(0,i-1)]$, it holds that $H^i(X_{\overline{k}},\mathbb{Q} _{\square}(r))^{G_k}=0$. 
  On the other hand, by the hard Lefschetz theorem, for every $i\geq d+1$, there is a $G_k$-isomorphism $H^i(X_{\overline{k}},\mathbb{Q} _{\square})\cong H^{2d-i}(X_{\overline{k}},\mathbb{Q} _{\square}(i-d))$. 
  Thus, for every $d+1\leq i< 2d $ and $r\notin [i-d,\ d-1]$, it holds that $H^i(X_{\overline{k}},\mathbb{Q} _{\square}(r))^{G_k}=0$. 
   Therefore, for every $i\neq 2d$ and $r\notin \left[\max(0, i - d), \, \max(0, \min(i - 1, d - 1))\right]$, we have  
    $H^i(X_{\overline{k}},\mathbb{Q} _{\square}(r))^{G_k}=0$. 
\end{proof}

\vspace{\baselineskip}

In the remainder of this section, we show the $p$-adic analogue of the above argument.\\ 

In the remainder of this section, we assume that $\Char k=0$ and that $\underline{k}$ is a perfect field of characteristic $p>0$.
Then we shall write 
$$K_0(\underline{k})\overset{\mathrm{def}}{=}W(\underline{k})[1/p]$$
for the fraction field of the Witt ring of $\underline{k}$.\\

After replacing $k$ by a finite extension, we may assume that $H^i(X_{\overline{k}},\mathbb{Q} _p)$ is a semi-stable representation.
Define
\[
D_{\textnormal{st}}(H^i(X_{\overline{k}}, \mathbb{Q} _p))\overset{\mathrm{def}}{=} \left(B_{\textnormal{st}}\otimes _{\mathbb{Q} _p}H^i(X_{\overline{k}}, \mathbb{Q} _p)\right)^{G_k}.
\]
Using the monodromy operator on $B_{\text{st}}$, we obtain a monodromy operator $N$ on $D_{\textnormal{st}}(H^i(X_{\overline{k}}, \mathbb{Q} _p))$.
Then $N$ determines the increasing filtration $M_{\bullet}$ on $D_{\textnormal{st}}(H^i(X_{\overline{k}}, \mathbb{Q} _p))$, called the monodromy filtration, characterized by the following properties:
\begin{enumerate}
  \item For every sufficiently large integer $j$, it holds that $M_j=D_{\textnormal{st}}(H^i(X_{\overline{k}}, \mathbb{Q} _p))$ and $M_{-j-1}=0$. 
  \item For every  integer $j$, it holds that $NM_j\subseteq  M_{j-2}$. 
  \item For every nonnegative integer $j$, the induced map $N^j:\gr ^M _jD_{\textnormal{st}}(H^i(X_{\overline{k}}, \mathbb{Q} _p))\to \gr ^M _{-j}D_{\textnormal{st}}(H^i(X_{\overline{k}}, \mathbb{Q} _p))$ is an isomorphism. 
\end{enumerate}
Note that if $X$ has semi-stable reduction, then, by the $C_{\text{st}}$-conjecture (cf.\ \cite[Theorem 0.2]{Tsuji}), an isomorphism
$D_{\textnormal{st}}(H^i(X_{\overline{k}}, \mathbb{Q} _p)) \cong H_{\textnormal{log-crys}}^i(\mathfrak{X} _{\underline{k}}/K_0(\underline{k}))$ holds.
Furthermore, if $k$ is an MLF, then Jannsen formulated the following conjecture as a $p$-adic analogue of Conjecture \ref{monodromy weight conjecture}:
\begin{conj}[$p$-adic monodromy weight conjecture, {\cite[p.\ 347]{Jannsen2}}]\label{p-adic monodromy weight conjecture} 
  As a $\varphi$-module, $\gr ^M _j D_{\textnormal{st}}(H^i(X_{\overline{k}}, \mathbb{Q} _p))$ is pure of weight $i+j$ and has the same eigenvalues as $\gr ^M _j H^i(X_{\overline{k}},\mathbb{Q} _{\ell})$ for $\ell $ ($\neq p$). 
\end{conj}

\begin{cor}[{\cite[Corollary 5.2]{Jannsen}}]\label{Jannsen Cor5.2}
  Let $k$ be an MLF. 
  Assume that Conjecture \ref{p-adic monodromy weight conjecture} holds for $X$. 
  Then, for every $r\notin \left[\max(0,i-d),\ \frac{i}{2}\right]$, we have  
  \[
  H^i(X_{\overline{k}},\mathbb{Q} _p(r))^{G_k}=0.
  \]
\end{cor}

\begin{thm}[{\cite[Theorem 5.3]{Jannsen}}]\label{Jannsen Thm5.3}
   Let $k$ be an MLF. 
   Then, for every $r\notin [\max (0, i-d),\, \min (i, d)]$, we have 
   $$H^i(X_{\overline{k}},\mathbb{Q} _p(r))^{G_k}=0.$$ 
\end{thm}

Assume that $X$ has strictly semi-stable reduction.
Then, for $\mathfrak{X} _{\underline{k}}$, there is the following $p$-adic weight spectral sequence (cf.\ \cite[\S 3.23]{Mokrane}, \cite[Part I. \S 2]{Nakkajima}):
\[
   E_1^{a,b}
   = \bigoplus_{r \ge \max(0, -a)}
   H_{\textnormal{crys}}^{b-2r} ( \mathfrak{X} _{\underline{k}} ^{(a+2r)}/K_0(\underline{k}))(-r)
  \Longrightarrow
  H_{\textnormal{log-crys}}^{a+b}( \mathfrak{X} _{\underline{k}} /K_0(\underline{k})).
  \]
In general, note that the $p$-adic weight spectral sequence exists for proper SNCL-variety over $\underline{k}$. 

If $k$ is an MLF, then it follows immediately from the finiteness of the residue field that the following properties hold:
\begin{enumerate}
  \item As a $\varphi$-module, $E_1^{a,b}$ is pure of weight $b$. 
        Thus, as a $\varphi$-module, $E_2^{a,b}$ is also pure of weight $b$. 
  \item The filtration on $H_{\textnormal{log-crys}}^i(\mathfrak{X} _{\underline{k}}/K_0(\underline{k}))$ induced by the convergence of the $p$-adic weight spectral sequence gives the weight filtration on $H_{\textnormal{log-crys}}^i(\mathfrak{X} _{\underline{k}}/K_0(\underline{k}))$. 
  \item The $p$-adic weight spectral sequence degenerates at the $E_2$-page. 
\end{enumerate}
Even if $k$ is not an MLF, by \cite[Theorem 3.6]{Nakkajima}, the $p$-adic weight spectral sequence degenerates at the $E_2$-page. 

Let
$$\varepsilon \overset{\mathrm{def}}{=}(\overline{\varepsilon}_n)_{n\in \mathbb{Z} _{\geq 0}}$$
be an element of $\mathcal{O} _{\mathbb{C} _p}^\flat$ satisfying $\varepsilon _0=1$, $\varepsilon _1\neq 1$, and $\varepsilon _{n+1}^p=\varepsilon _n$, where $\overline{\varepsilon}_n$ is the reduction of $\varepsilon_n$ modulo $p$ for each $n \in \mathbb{Z}_{\geq 0}$.
Let
\[[\,\text{-}\,]:\mathcal{O} _{\mathbb{C} _p}^\flat \longrightarrow A_{\text{inf}}\]
be the Teichm\"uller lift, and define
\[t\overset{\mathrm{def}}{=}\log [\varepsilon ]=\sum_{m=1}^{\infty}(-1)^{m-1}\frac{([\varepsilon ]-1)^m}{m}\in B_{\text{dR}}.\]
Let $e$ be a generator of $\mathbb{Q} _p(1)$.
Then 
\[
t^{-1}\otimes e : K_0(\underline{k}) \longrightarrow K_0(\underline{k})(1) \overset{\mathrm{def}}{=} (B_{\textnormal{crys}}\otimes _{\mathbb{Q} _p}\mathbb{Q}_p(1))^{G_k} 
\]
is an isomorphism. 
With this notation, we obtain the following morphism between the $p$-adic weight spectral sequences:
\[
\begin{tikzcd}
E_1^{a,b}
=
\displaystyle
\bigoplus_{r \ge \max(0,-a)}
H_{\mathrm{crys}}^{\,b-2r}
(
\mathfrak{X}_{\underline{k}}^{(a+2r)}/K_0(\underline{k})
)
(-r)
\arrow[r, Rightarrow]
\arrow[d, "\mathrm{id}\otimes (t^{-1}\otimes e)"]
&
H_{\mathrm{log\text{-}crys}}^{a+b}
(
\mathfrak{X}_{\underline{k}}/K_0(\underline{k})
)
\arrow[d, "N"]
\\
E_1^{a+2,b-2}
=
\displaystyle
\bigoplus_{r \ge \max(0,-a)}
H_{\mathrm{crys}}^{\,b-2r}
(
\mathfrak{X}_{\underline{k}}^{(a+2r)}/K_0(\underline{k})
)
(-r+1)
\arrow[r, Rightarrow]
&
H_{\mathrm{log\text{-}crys}}^{a+b}
(
\mathfrak{X}_{\underline{k}}/K_0(\underline{k})
).
\end{tikzcd}
\]

\begin{prop}\label{p-adic isom of N^d}
   Assume that $d\geq 1$ and $X$ has strictly semi-stable reduction. 
   Then $N^d:E_2^{-d,2d}\to E_2^{d,0}$ is an isomorphism. 
\end{prop}
\begin{proof}
  For simplicity, we omit Tate twists and coefficients. 
  It suffices to prove that the following natural map is an isomorphism:
  \[
  \Ker \left(H_{\text{crys}}^0(\mathfrak{X} _{{\underline{k}}}^{(d)})\xrightarrow{\mathrm{Gys}}H_{\text{crys}}^2(\mathfrak{X} _{{\underline{k}}}^{(d-1)})\right)\longrightarrow \Coker \left(H_{\text{crys}}^0(\mathfrak{X} _{{\underline{k}}}^{(d-1)})\xrightarrow{\mathrm{Res}}H_{\text{crys}}^0(\mathfrak{X} _{{\underline{k}}}^{(d)})\right)
  \]
  --- where $\mathrm{Gys}$ (resp.\ $\mathrm{Res}$) is the sum of the Gysin maps (resp.\ restriction maps), each multiplied by an appropriate sign $(\pm 1)$. 
  Since $\mathrm{Gys}$ and $\mathrm{Res}$ are dual to each other, it is enough to prove that the above map is injective.  
  Therefore, it suffices to prove that the nondegenerate bilinear pairing induced by the cup product
  $
  B: H_{\text{crys}}^0(\mathfrak{X} _{{\underline{k}}}^{(d)}) \times H_{\text{crys}}^0(\mathfrak{X} _{{\underline{k}}}^{(d)}) \longrightarrow 
  K_0(\underline{k}) 
  $
  remains nondegenerate when restricted to $\im \mathrm{Res}\times \im \mathrm{Res}$. 
  Since each of $\square \in \{H_{\text{crys}}^0(\mathfrak{X} _{{\underline{k}}}^{(d)}), H_{\text{crys}}^0(\mathfrak{X} _{{\underline{k}}}^{(d-1)}), \mathrm{Res}, B \}$ admits the corresponding combinatorial $\mathbb{Q} $-structure $\square_{\mathbb{Q}} $, the proof of the assertion is similar to that of Proposition \ref{isom of N^d}. 
\end{proof}

\begin{prop}\label{wt of kerN}
  Let $k$ be an MLF.
  Assume that $d\geq 1$ and $X$ has strictly semi-stable reduction. 
  Then, as a $\varphi$-module, $\Ker \left(N:H_{\textnormal{log-crys}}^d(\mathfrak{X} _{\underline{k}}/K_0(\underline{k}))\to H_{\textnormal{log-crys}}^d(\mathfrak{X} _{\underline{k}}/K_0(\underline{k}))\right)$ is mixed with weights in $[0,\, 2d-1]$. 
\end{prop}
\begin{proof}
  We consider the following $p$-adic weight spectral sequence:
  \[
   E_1^{a,b}
   = \bigoplus_{r \ge \max(0, -a)}
   H_{\textnormal{crys}}^{b-2r} ( \mathfrak{X} _{\underline{k}} ^{(a+2r)}/K_0(\underline{k}))(-r)
  \Longrightarrow
  H_{\textnormal{log-crys}}^{a+b}( \mathfrak{X} _{\underline{k}} /K_0(\underline{k})).
  \]
  By Proposition \ref{p-adic isom of N^d}, we have  
  $$\Ker \left(H_{\textnormal{log-crys}}^d(\mathfrak{X} _{\underline{k}}/K_0(\underline{k}))\xrightarrow{N} H_{\textnormal{log-crys}}^d(\mathfrak{X} _{\underline{k}}/K_0(\underline{k}))\right)\subseteq  E_2^{-d+1,2d-1}\oplus E_2^{-d+2,2d-2}\oplus \cdots \oplus E_2^{d,0} .$$
  Since  $E_2^{a,b}$ is pure of weight $b$, $\Ker \left(N^d:H_{\textnormal{log-crys}}^d(\mathfrak{X} _{\underline{k}}/K_0(\underline{k}))\to H_{\textnormal{log-crys}}^d(\mathfrak{X} _{\underline{k}}/K_0(\underline{k}))\right)$ is mixed with weights in $[0,\, 2d-1]$ as $\varphi$-modules. 
\end{proof}

The functor $D_{\mathrm{st}}$ is a fully faithful functor from the category of semi-stable $G_k$-representations over $\mathbb{Q}_p$ to the category of filtered $(\varphi,N)$-modules over $k$. 
Furthermore, it is an equivalence of categories if the target is restricted to the full subcategory of admissible filtered $(\varphi,N)$-modules.

\begin{thm}\label{p-adic vanishing over MLF}
    Let $k$ be an MLF. 
    Then we have 
    $$H^i(X_{\overline{k}},\mathbb{Q} _p(r))^{G_k}=0$$
    for every $i\neq 2d$ and $r\notin \left[\max(0, i - d), \, \max(0, \min(i - 1, d - 1))\right]$.
    In particular, for every $i\neq 0$, we have  
    $H^i(X_{\overline{k}},\mathbb{Q} _p)_{G_k}=0$.
\end{thm}
\begin{proof}
  By the $C_{\text{st}}$-conjecture, we have 
  \[
  \begin{aligned}
    H^d(X_{\overline{k}},\mathbb{Q} _p)^{G_k} &\cong \Hom _{G_k}\left(\mathbb{Q} _p, H^d(X_{\overline{k}},\mathbb{Q} _p)\right)\\
                                                 &\cong \Hom _{\varphi, N, \mathrm{Fil}^{\bullet}} \left(K_0(\underline{k}), D_{\textnormal{st}}(H^d(X_{\overline{k}},\mathbb{Q} _p))\right)\\
                                                 &\cong \Hom _{\varphi, N, \mathrm{Fil}^{\bullet}} \left(K_0(\underline{k}), H_{\textnormal{log-crys}}^d( \mathfrak{X} _{\underline{k}} /K_0)\right)\\
                                                 &\subseteq  \Ker \left( H_{\textnormal{log-crys}}^d( \mathfrak{X} _{\underline{k}} /K_0(\underline{k})) \xrightarrow{N} H_{\textnormal{log-crys}}^d( \mathfrak{X} _{\underline{k}} /K_0(\underline{k})) \right)^{\varphi =1}\\                                                 
  \end{aligned}
  \]
  --- where the last inclusion is obtained by evaluating each morphism at $1\in K_0(\underline{k})$. 
  Hence, by Proposition \ref{wt of kerN}, for every $r\notin [0,\, d-1]$, it holds that  $H^d(X_{\overline{k}},\mathbb{Q} _p(r))^{G_k}=0$. 
  Thus, by Lemma \ref{key lem of vanishing}, for every $i\neq 2d$ and $r\notin \left[\max(0, i - d), \, \max(0, \min(i - 1, d - 1))\right]$, we have 
    $H^i(X_{\overline{k}},\mathbb{Q} _p(r))^{G_k}=0.$
\end{proof}

\begin{prop}\label{p-adic isom of N^d-1}
   Let $k$ be an MLF.
   Assume that $d\geq 2$ and $X$ has strictly semi-stable reduction. 
   Then $N^{d-1}:E_2^{-d+1,2d-1}\to E_2^{d-1,1}$ is an isomorphism. 
\end{prop}
\begin{proof}
  For simplicity, we omit Tate twists and coefficients. 
  It suffices to prove that the following natural map is an isomorphism:
  \[
  \Ker \left(H_{\text{crys}}^1(\mathfrak{X} _{{\underline{k}}}^{(d-1)})\xrightarrow{\mathrm{Gys}}H_{\text{crys}}^3(\mathfrak{X} _{{\underline{k}}}^{(d-2)})\right)\longrightarrow \Coker \left(H_{\text{crys}}^1(\mathfrak{X} _{{\underline{k}}}^{(d-2)})\xrightarrow{\mathrm{Res}}H_{\text{crys}}^1(\mathfrak{X} _{{\underline{k}}}^{(d-1)})\right)
  \]
  --- where $\mathrm{Gys}$ (resp.\ $\mathrm{Res}$) is the sum of the Gysin maps (resp.\ restriction maps), each multiplied by an appropriate sign $(\pm 1)$.  
  Thus, the assertion follows from \cite[Lemma 6.2.1]{Mokrane}. 
\end{proof}

\begin{prop}
  Let $k$ be an MLF. 
  Assume that $d\geq 2$ and $X$ has strictly semi-stable reduction.  
  Then, as a $\varphi$-module, $\Ker \left(N:H_{\textnormal{log-crys}}^d(\mathfrak{X} _{\underline{k}}/K_0(\underline{k}))\to H_{\textnormal{log-crys}}^d(\mathfrak{X} _{\underline{k}}/K_0(\underline{k}))\right)$ is mixed with weights in $[0,\, 2d-2]$. 
\end{prop}
\begin{proof}
  We consider the following $p$-adic weight spectral sequence:
  \[
   E_1^{a,b}
   = \bigoplus_{r \ge \max(0, -a)}
   H_{\textnormal{crys}}^{b-2r} ( \mathfrak{X} _{\underline{k}} ^{(a+2r)}/K_0(\underline{k}))(-r)
  \Longrightarrow
  H_{\textnormal{log-crys}}^{a+b}( \mathfrak{X} _{\underline{k}} /K_0(\underline{k})).
  \] 
  By Proposition \ref{p-adic isom of N^d} and \ref{p-adic isom of N^d-1}, we have 
  $$\Ker \left(H_{\textnormal{log-crys}}^d(\mathfrak{X} _{\underline{k}}/K_0(\underline{k}))\xrightarrow{N} H_{\textnormal{log-crys}}^d(\mathfrak{X} _{\underline{k}}/K_0(\underline{k}))\right)\subseteq  E_2^{-d+2,2d-2}\oplus E_2^{-d+3,2d-3}\oplus \cdots \oplus E_2^{d,0} .$$
  Since $E_2^{a,b}$ is pure of weight $b$,
  $\Ker \left(N:H_{\textnormal{log-crys}}^d(\mathfrak{X} _{\underline{k}}/K_0(\underline{k}))\to H_{\textnormal{log-crys}}^d(\mathfrak{X} _{\underline{k}}/K_0(\underline{k}))\right)$ is mixed with weights in $[0,\, 2d-2]$ as $\varphi$-modules. 
\end{proof}

\vspace{\baselineskip}
\section{Properties of Kummer-faithful fields}

In this section, we introduce Kummer-faithful fields, together with several variant notions of them, and review some of the properties they satisfy.

\begin{defn}[{\cite[Definition 1.5]{MochizukiIII}}, {\cite[Definition 1.2]{Hoshi}}, {\cite[Definition 3.1]{Murotani2}}]\label{Def of KF}
  Let $K$ be a field. 
  \begin{enumerate}[label=(\roman*)]
    \item If, for any finite extension $L/K$ and semi-abelian variety (resp.\ abelian variety, resp.\ torus) $A$ over $L$, it holds that
          \[
          A(L)_{\text{div}}=0,
          \]
          we shall say that $K$ is \textbf{pre-Kummer-faithful} (resp.\ \textbf{pre-AVKF}, resp.\ \textbf{pre-torally Kummer-faithful})
          (where ``AVKF" is understood as an abbreviation for ``Abelian Variety Kummer-Faithful"). 
    \item We shall say that $K$ is \textbf{Kummer-faithful} (resp.\ \textbf{AVKF}, resp.\ \textbf{torally Kummer-faithful}), if $K$ is perfect and pre-Kummer-faithful (resp.\ pre-AVKF, resp.\ pre-torally Kummer-faithful).
  \end{enumerate}
\end{defn}

It follows immediately from  the definitions of a Kummer-faithful field (resp.\ an AVKF-field, resp.\ a torally Kummer-faithful field) that the following assertions hold:
\begin{enumerate}
  \item If a perfect field $K$ is Kummer-faithful (resp.\ AVKF, resp.\ torally Kummer-faithful), then so is any perfect subfield of $K$. 
  \item Let $L/K$ be finite extension of a perfect field. 
        Then $K$ is Kummer-faithful (resp.\ AVKF, resp.\ torally Kummer-faithful) if and only if $L$ is Kummer-faithful (resp.\ AVKF, resp.\ torally Kummer-faithful). 
\end{enumerate}
By \cite[Remark 1.5.2]{MochizukiIII}, in the definition of ``Kummer-faithful", ``AVKF", and ``torally Kummer-faithful", it is equivalent to require only the case where $L=K$.
Furthermore, by \cite[Remark 1.5.4]{MochizukiIII}, any sub-$p$-adic field is Kummer-faithful. 
Note that a perfect field $K$ is Kummer-faithful if and only if $K$ is both AVKF and torally Kummer-faithful (cf.\ \cite[Proposition 2.3]{Ozeki-Taguchi}).

\vspace{\baselineskip}
In the remainder of this section, let $K$ be a field.
Then, for $p \in \mathfrak{Primes}$ and $n\in \mathbb{Z} _{>0}$, we shall write 
\begin{itemize}
  \item $K^{\times}\overset{\mathrm{def}}{=}K\backslash \{0\}$,
  \item $\mu _n(K)\overset{\mathrm{def}}{=}\{x\in K^{\times}| x^n=1\}$, 
  \item $\mu _{p^{\infty}}(K)\overset{\mathrm{def}}{=}\bigcup _{m\in \mathbb{Z} _{>0}}\mu _{p^m}(K)$.
\end{itemize} 
\vspace{\baselineskip}
\begin{defn}[{\cite[Definition 6.1]{Hoshi-Mochizuki-Tsujimura}}]\label{Def of stably mu-finite}
  Let $\ell \in \mathfrak{Primes}$. 
  If, for any finite extension $L/K$, $\mu _{\ell ^{\infty}}(L)$ is finite, we shall say that $K$ is \textbf{stably $\mu _{\ell ^{\infty}}$-finite}. 
\end{defn}

\begin{defn}[{\cite[Definition 6.1]{Hoshi-Mochizuki-Tsujimura}}, {\cite[Definition 2.5]{Ozeki}}]\label{Def of semi-AV-tor-finite}
  Let $K$ be a perfect field. 
  Let $\ell \in \mathfrak{Primes}$, and let $\Sigma$ be a nonempty subset of $\mathfrak{Primes}$. 
  \begin{enumerate}
    \item If, for any finite extension $L/K$ and any semi-abelian variety (resp.\ abelian variety) $A$ over $L$, $A(L)[\ell ^{\infty}]$ is finite, we shall say that $K$ is \textbf{$\ell ^{\infty}$-semi-AV-tor-finite} (resp.\ \textbf{$\ell ^{\infty}$-AV-tor-finite}). 
    \item If, for every $\ell \in \Sigma $, $K$ is $\ell ^{\infty}$-semi-AV-tor-finite (resp.\ $\ell ^{\infty}$-AV-tor-finite), we shall say that $K$ is \textbf{$\Sigma ^{\infty}$-semi-AV-tor-finite} (resp.\ \textbf{$\Sigma ^{\infty}$-AV-tor-finite}). 
    \item If, for any finite extension $L/K$ and any semi-abelian variety (resp.\ abelian variety) $A$ over $L$, $A(L)_{\text{tor}}$ is finite, we shall say that $K$ is \textbf{semi-AV-tor-finite} (resp.\ \textbf{AV-tor-finite}). 
  \end{enumerate}
\end{defn}
The following proposition is useful for studying the notions introduced in Definition \ref{Def of semi-AV-tor-finite}:
\begin{prop}[{\cite[Proposition 2.4]{Ozeki-Taguchi}}]\label{Ozeki-Taguchi, Prop. 2,4}
  Let $k$ be a field, $A$ a semi-abelian variety over $k$, and $K/k$ an algebraic extension. 
  Then, for each $\ell \in \mathfrak{Primes}$, the following conditions are equivalent:
  \begin{enumerate}[label=(\roman*)]
    \item It holds that $(A(K)[\ell ^{\infty}])_{\ell {\textnormal{-div}}}=0$.
    \item $A(K)[\ell^{\infty}]$ is finite. 
    \item It holds that $V_{\ell}(A)^{G_K}=0$.
  \end{enumerate}
  Moreover, we consider the following conditions:
  \begin{enumerate}[label=(\arabic*)]
    \item It holds that $A(K)_{\textnormal{div}}=0$.
    \item It holds that $(A(K)_{\textnormal{tor}})_{\textnormal{div}}=0$.
    \item For each $\ell \in \mathfrak{Primes}$, $A(K)[\ell ^{\infty}]$ is finite. 
  \end{enumerate}
  Then,
  \begin{itemize}
    \item $(1)\Rightarrow  (2)\Leftrightarrow (3)$ holds,
    \item $(1)\Leftrightarrow   (2)\Leftrightarrow (3)$ holds if $k$ is Kummer-faithful and $K/k$ is a Galois extension.
  \end{itemize}
\end{prop}

Let $\square$ be a one of the following properties:
\begin{itemize}
  \item stably $\mu _{\ell ^{\infty}}$-finite,
  \item $\ell ^{\infty}$-semi-AV-tor-finite,
  \item $\ell ^{\infty}$-AV-tor-finite,
  \item $\Sigma ^{\infty}$-semi-AV-tor-finite,
  \item $\Sigma ^{\infty}$-AV-tor-finite,
  \item semi-AV-tor-finite,
  \item AV-tor-finite.
\end{itemize}
Then it follows immediately from the definitions of $\square$ that the following assertions hold:
\begin{enumerate}
  \item If a perfect field $K$ is $\square$, then so is any perfect subfield of $K$. 
  \item Let $L/K$ be a finite extension of a perfect field. 
        Then $K$ is $\square$ if and only if $L$ is $\square$. 
\end{enumerate}
If $\square$ is not stably $\mu_{\ell^\infty}$-finite, then, as in \cite[Remark~1.5.4]{MochizukiIII}, it is equivalent in the definition of the $\square$-property to require only the case where $L=K$.
For a perfect field $K$ and a prime number $\ell \in \mathfrak{Primes}$, we have $\mu _{\ell ^{\infty}}(K)=\mathbb{G} _m(K)[\ell ^{\infty}]$.
Hence, if $K$ is $\ell ^{\infty}$-semi-AV-tor-finite, then $K$ is stably $\mu _{\ell ^{\infty}}$-finite. 
Furthermore, any semi-AV-tor-finite (resp.\ AV-tor-finite) field is $\mathfrak{Primes}^{\infty}$-semi-AV-tor-finite (resp.\ $\mathfrak{Primes}^{\infty}$-AV-tor-finite). 
By \cite[Proposition 2.9]{Ozeki-Taguchi}, any sub-$p$-adic field is semi-AV-tor-finite. 
If a perfect field $K$ is Kummer-faithful, then $K$ is $\mathfrak{Primes}^{\infty}$-semi-AV-tor-finite.
Furthermore, if $K$ is a Galois extension of Kummer-faithful field, then the converse also holds by Proposition \ref{Ozeki-Taguchi, Prop. 2,4}. 
Finally, by \cite[Theorem 3.9]{Ozeki}, for a Galois extension of an MLF, a field is  semi-AV-tor-finite if and only if it is Kummer-faithful and has finite residue field.
In addition, the following criterion for $\ell^{\infty}$-semi-AV-tor-finiteness holds:

\begin{prop}[cf.\ {\cite[Proposition 2.3]{Ozeki-Taguchi}}]\label{finite=T-finite+A-finite}
  The following conditions are equivalent:
  \begin{enumerate}[label=(\roman*)]
    \item For any semi-abelian variety $B$ over $K$, it holds that $B(K)[\ell ^{\infty}]_{\ell \textnormal{-div}}=0$.
    \item For any finite extension $L/K$, it holds that $\mathbb{G} _m(L)[\ell ^{\infty}]_{\ell \textnormal{-div}}=0$, and for any abelian variety $A$ over $K$, it holds that $A(K)[\ell ^{\infty}]_{\ell \textnormal{-div}}=0$.
  \end{enumerate}
\end{prop}
\begin{proof}
  Since $(i)\Rightarrow (ii)$ is clear, we prove $(ii)\Rightarrow (i)$. 
  Let $B$ be a semi-abelian variety over $K$, given as an extension of an abelian variety $A$ by a torus $T$. 
  Assume that $B(K)[\ell ^{\infty}]_{\ell \text{-div}}$ is nonzero, and choose $P_0\in B(K)[\ell ^{\infty}]_{\ell \text{-div}}\backslash \{0\}$. 
  For each $n\in \mathbb{Z} _{>0}$, define $X_{\ell ^n} \overset{\mathrm{def}}{=} \{P_n \in B(K)[\ell ^{\infty}]| \ell ^n P_n=P_0\}$. 
  Then $X_{\ell ^n}$ is a nonempty finite set. 
  For $n,m\in \mathbb{Z} _{>0}$ with $m\geq n$, define $f_{n,m}:X_{\ell ^m}\to X_{\ell ^n}$, $P_m\mapsto \ell ^{m-n}P_m$.
  Then $\{X_{\ell ^n}\}_{n\in \mathbb{Z} _{>0}}$ forms an inverse system, and hence its inverse limit $\varprojlim X_{\ell^n}$ is nonempty. 
  Take an element $(P_n)_{n \in \mathbb{Z} _{>0}} \in \varprojlim X_{\ell ^n}$. 
  Let $\overline{P}_n$ be the image of $P_n$ under the map $B(K)[\ell ^{\infty}]\to A(K)[\ell ^{\infty}]$. 
  Since $A(K)[\ell ^{\infty}]_{\ell \text{-div}}=0$, we have $\overline{P}_n=0$ for each $n$. 
  Thus, $P_1$ is a nonzero divisible element of $T(K)[\ell ^{\infty}]$. 
  On the other hand, assume that $T$ splits over a finite extension $L/K$.
  Then $T(K)$ may be identified with a subset of $\mathbb{G} _m(L)^{\dim T}$. 
  Thus, $T(K)[\ell ^{\infty}]_{\ell \text{-div}}=0$, which contradicts the fact that $P_1$ is a nonzero divisible element of $T(K)[\ell ^{\infty}]$. 
\end{proof}

\vspace{\baselineskip}
\section{Criterion for (quasi-)highly Kummer-faithful fields}

In this section, we study the (quasi-)high Kummer-faithfulness of various fields. 
In particular, we show that sub-$p$-adic fields are quasi-highly Kummer-faithful. 
Furthermore, we  give equivalent conditions for the vanishing of the coinvariants of \'etale cohomology in the case of algebraic extensions of FF's and CDVF's whose residue fields are algebraic extensions of some FF. 
As a result, we establish equivalences between (torally) Kummer-faithfulness and (quasi-)high Kummer-faithfulness.\\

In the remainder of this section, let $r$ be an integer, $d$ a nonnegative integer, and $i$ an integer with $0\le i\le 2d$ for a $d$-dimensional proper smooth variety $X$.\\

\begin{defn}[{\cite[Definition 2.6]{Ozeki-Taguchi}}]\label{Def of HKF}
    Let $k$ be a perfect field, and set $p_k\overset{\mathrm{def}}{=}\Char k$. 
    For $k$, consider the following condition:
    \begin{itemize}
    \item[$(\dagger)_{d,i,r}$] For any finite extension $k_H/k$, any $d$-dimensional proper smooth variety $X$ over $k_H$, each $\ell \in \mathfrak{Primes}\backslash \{p_k\}$, it holds that 
    $$H^i(X_{\overline{k}_H},\mathbb{Q}_{\ell}(r))_{H}=0$$ 
    --- where $H\overset{\mathrm{def}}{=}G_{k_H}$. 
    \end{itemize}
    \begin{enumerate}[label=(\roman*)]
        \item If, for every $0<i\leq 2d$, the field $k$ satisfies the condition $(\dagger)_{d,i,0}$, we shall say that $k$ is \textbf{quasi-highly Kummer-faithful}. 
        \item If, for every $d$ and every pair of integers $i,r$ with $i\neq 2r$, the field $k$ satisfies the condition $(\dagger)_{d,i,r}$, we shall say that $k$ is \textbf{highly Kummer-faithful}.  
    \end{enumerate}
\end{defn}

It follows immediately from the definition that 
\begin{center}
  $k$ is highly Kummer-faithful $\Longrightarrow $ $k$ is quasi-highly Kummer-faithful.
\end{center}
By Poincar\'e duality, the condition $(\dagger)_{d,i,r}$ is also equivalent to the following condition:
\begin{center} 
  ``For any finite extension $k_H/k$, any $d$-dimensional proper smooth variety $X$ over $k_H$, each $\ell \in \mathfrak{Primes}\backslash \{p_k\}$, it holds that $H^{2d-i}(X_{\overline{k}_H},\mathbb{Q}_{\ell}(d-r))^{H}=0$, where $H\overset{\mathrm{def}}{=}G_{k_H}$. "
\end{center} 
It follows immediately from the definition of a highly Kummer-faithful field (resp.\ quasi-highly Kummer-faithful field) that the following assertions hold:
\begin{enumerate}
  \item If a perfect field $k$ is highly Kummer-faithful (resp.\ quasi-highly Kummer-faithful), then so is any perfect subfield of $k$. 
  \item Let $k_H/k$ be a finite extension of a perfect field. 
  Then $k$ is highly Kummer-faithful (resp.\ quasi-highly Kummer-faithful) if and only if $k_H$ is highly Kummer-faithful (resp.\ quasi-highly Kummer-faithful). 
\end{enumerate}
Assume that a field $k$, not necessarily perfect, satisfies the condition $(\dagger)_{d,i,r}$.
Then, since $G_k$ and $G_{k^{\text{perf}}}$ coincide, the perfect closure $k^{\text{perf}}$ of $k$ also satisfies the condition $(\dagger)_{d,i,r}$.
Hence, for any integers $d(\geq 0), i, r$ satisfying $i\neq 2r$ (resp.\ $0< i\leq 2d$), if $k$ satisfies the condition $(\dagger)_{d,i,r}$ (resp.\ ($\dagger ) _{d,i,0}$), then $k^{\text{perf}}$ is highly Kummer-faithful (resp.\ quasi-highly Kummer-faithful).
Assume that a perfect field $k$ is quasi-highly Kummer-faithful and is a Galois extension of a Kummer-faithful field.
Then, by \cite[Proposition 2.9]{Ozeki-Taguchi}, the field $k$ is Kummer-faithful.
The notions of (quasi-)highly Kummer-faithful fields were introduced in order to characterize Kummer-faithful fields in terms of Galois representations.
It is therefore natural to investigate similarities between (quasi-)highly Kummer-faithful fields and Kummer-faithful fields.
As a first basic test case for measuring such similarities, we consider sub-$p$-adic fields.
However, it is known that $\mathbb{Q}_p$ is not highly Kummer-faithful (cf.\ \cite[p.\ 54]{Ozeki-Taguchi}).
Thus, Ozeki and Taguchi posed the following question:

\begin{ques}[{\cite[Question]{Ozeki-Taguchi}}]\label{ques}
  Let $k$ be a sub-$p$-adic field.
  Then is $k$ quasi-highly Kummer-faithful?
\end{ques}

\begin{prop}[{\cite[Proposition 2.11]{Ozeki-Taguchi}}]
  Let $k$ be an MLF.
  Assume that Conjectures \ref{monodromy weight conjecture} and \ref{p-adic monodromy weight conjecture} hold for any proper smooth variety over any finite extension of $k$ with semi-stable reduction. 
  Then $k$ is quasi-highly Kummer-faithful. 
\end{prop}
\begin{proof}
  The assertion follows from Corollaries \ref{Jannsen Cor4.3} and \ref{Jannsen Cor5.2} together with de Jong's alteration. 
\end{proof}
\vspace{\baselineskip}
In the remainder of this section, we show (quasi-)highly Kummer-faithfulness for various fields.
In Theorem \ref{QHKFness of sub-p-adic fields}, we give an affirmative answer to Question \ref{ques} without assuming Conjectures \ref{monodromy weight conjecture} and \ref{p-adic monodromy weight conjecture}.

\begin{prop}
  Let $k$ be an FF.
  Then $k$ is highly Kummer-faithful. 
\end{prop}
\begin{proof}
  Let $\ell \in \mathfrak{Primes}\backslash \{\Char k\}$ and let $X$ be a proper smooth variety over $k$. 
  Then the $G_k$-representation $H^i(X_{\overline{k}},\mathbb{Q} _{\ell})$ is pure of weight $i$. 
\end{proof} 

\begin{prop}
  Let $k$ be a PLF.
  Then $k^{\textnormal{perf}}$ is quasi-highly Kummer-faithful. 
\end{prop}
\begin{proof}
  By Remark \ref{remark of monodromy weight conjecture} and Corollary \ref{Jannsen Cor4.3}, for every $r\notin \left[\max(0, i-d),\, \frac{i}{2}\right]$, the field $k$ satisfies the condition $(\dagger)_{d,i,r}$. 
  Thus, $k^{\text{perf}}$ is quasi-highly Kummer-faithful. 
\end{proof}

\begin{prop}\label{f.g. over HKF}
   Let $K$ be a finitely generated field extension of a field $k$, $X$ a $d$-dimensional proper smooth variety over $K$, and $\ell \in \mathfrak{Primes}\backslash \{\Char k\}$.
   If $k$ satisfies the condition $(\dagger)_{d,i,r}$, then we have 
   $$H^i(X_{\overline{K}},\mathbb{Q} _\ell(r))_{G_{K}}=0.$$
   In particular, if $k$ is highly Kummer-faithful (resp.\ quasi-highly Kummer-faithful), then $K^{\textnormal{perf}}$ is highly Kummer-faithful (resp.\ quasi-highly Kummer-faithful). 
\end{prop}
\begin{proof}
  There exists a normal integral scheme $S$ over $k$ whose function field $k(S)$ coincides with $K$.
  By spreading out, there exist an open subscheme $U \subseteq S$ and a model $\mathfrak{X}$ over $U$ such that the structure morphism $f_U:\mathfrak{X}\to U$ is proper, smooth, and geometrically irreducible, and such that the generic fiber of $\mathfrak{X}$ is isomorphic to $X$.
  Let $s$ be a closed point of $U$ and $\overline{s}$ a geometric point over $s$. 
  Note that the residue field $\kappa(s)$ of $s$ is a finite extension of $k$.
  By the proper base change theorem, there is a canonical isomorphism  $({R^i{f_U} _*}\mathbb{Q} _{\ell})_{\overline{s}} \cong H^i(\mathfrak{X} _{\overline{\kappa (s)}},\mathbb{Q} _{\ell})$, where  $({R^i{f_U} _*}\mathbb{Q} _{\ell})_{\overline{s}}$ is the stalk of ${R^i{f_U} _*}\mathbb{Q} _{\ell}$ at $\overline{s}$.
  Let $\pi ^{\text{\'et}} _1(U,\overline{s})$ be the \'etale fundamental group of $U$ with basepoint $\overline{s}$.
  Since ${R^i{f_U} _*}\mathbb{Q} _{\ell}$ is a lisse sheaf, the stalk $({R^i{f_U} _*}\mathbb{Q} _{\ell})_{\overline{s}}$ is $\pi ^{\text{\'et}} _1(U,\overline{s})$-representation, and the action of $G_{\kappa (s)}$ (resp.\ $G_K$) on $H^i(\mathfrak{X} _{\overline{\kappa (s)}},\mathbb{Q} _{\ell})$ (resp.\ $H^i(X _{\overline{K}},\mathbb{Q} _{\ell})$) may be regarded as factoring through $\pi ^{\text{\'et}} _1(U,\overline{s})$. 
  Since $U$ is normal, the natural morphism $G_K\to \pi ^{\text{\'et}} _1(U,\overline{s})$ is surjective. 
  Thus, we have $H^i(X_{\overline{K}},\mathbb{Q} _\ell(r))_{G_{K}}=0$. 
\end{proof}

\begin{thm}\label{QHKFness of sub-p-adic fields}
    Let $k$ be a sub-$p$-adic field.
    Then $k$ is quasi-highly Kummer-faithful. 
\end{thm}
\begin{proof}
  By Theorems \ref{vanishing over MLF} and \ref{p-adic vanishing over MLF}, MLF's are quasi-highly Kummer-faithful. 
  Note that any perfect subfield of quasi-highly Kummer-faithful field is quasi-highly Kummer-faithful.
  Thus, by Proposition \ref{f.g. over HKF}, $k$ is quasi-highly Kummer-faithful.  
\end{proof}
\vspace{\baselineskip}
 In the remainder of this section, we give equivalent conditions for the vanishing of the coinvariants  of \'etale cohomology over algebraic extensions of some FF and over CDVF's whose residue fields are algebraic extensions of some FF.
 As a result, we show the equivalence between (torally) Kummer-faithful fields and (quasi-)highly Kummer-faithful fields.

 \begin{lem}\label{key lem of Gal rep}
    Let $\ell \in \mathfrak{Primes}$, and let $V$ be a $G_{\mathbb{F} _q}$-module over $\mathbb{Q} _{\ell}$. 
    Let $\mathbb{F} /\mathbb{F} _q$ be an algebraic extension, and assume that $G_{\mathbb{F} }^{\ell}$ is nontrivial. 
    If, for any finite extension $\mathbb{F} _{q^n}/\mathbb{F} _q$, it holds that $V^{G_{\mathbb{F} _{q^n}}}=0$, then we have $$V^{G_{\mathbb{F} }}=0.$$      
\end{lem}
\begin{proof}
  Let $n$ be the dimension of $V$, and fix an isomorphism $\Aut (V)\cong \GL _n(\mathbb{Q} _{\ell})$. 
  Then we may assume that the action $G_{\mathbb{F} }\to \GL _n(\mathbb{Q} _{\ell})$ factors through $\GL _n(\mathbb{Z} _{\ell})$. 
  By replacing $\mathbb{F}$ with the finite extension corresponding to the kernel of the composite homomorphism with the reduction map $G_{\mathbb{F} }\to \GL _n(\mathbb{Z} _{\ell})\to \GL _n(\mathbb{F} _{\ell})$, we may assume that the action $G_{\mathbb{F} }\to \GL _n(\mathbb{Q} _{\ell})$ factors through $\mathbb{Z} _{\ell}$ surjectively. 
  Similarly, by replacing $\mathbb{F}_q$ with a finite extension $\mathbb{F}_{q^m}$ if necessary, we may assume that the action $G_{\mathbb{F} _{q^m}}\to \GL _n(\mathbb{Q} _{\ell})$ factors through $\mathbb{Z} _{\ell}$ surjectively. 
  Therefore, by replacing $\mathbb{F}$ once again by a finite extension if necessary, we may assume that the actions of $G_{\mathbb{F}}$ and $G_{\mathbb{F}_{q^m}}$ on $V$ coincide. 
  Thus, we have 
  $$V^{G_\mathbb{F} }=V^{\mathbb{G} _{\mathbb{F} _{q^m}}}=0.$$ 
\end{proof}

\begin{prop}\label{vanishing of coinvariants over alg ext over FF}
  For $p\in \mathfrak{Primes}$, let $\mathbb{F} /\mathbb{F} _p$ be an algebraic extension, and let $\ell \in \mathfrak{Primes}\backslash \{p\}$. 
  Then the following conditions are equivalent:
  \begin{enumerate}[label=(\roman*)]
    \item $G_{\mathbb{F} }^{\ell}$ is nontrivial.
    \item For any proper smooth variety $X$ over $\mathbb{F} $ and every pair of integers $i,r$ with $i\neq 2r$, it holds that 
          $H^i(X_{\overline{\mathbb{F} }},\mathbb{Q} _{\ell}(r))^{G_{\mathbb{F} }}=0$. 
    \item For any proper smooth variety $X$ over $\mathbb{F} $ and every integer $i\neq 0$, it holds that $H^i(X_{\overline{\mathbb{F} }},\mathbb{Q} _{\ell})_{G_{\mathbb{F} }}=0$. 
  \end{enumerate}
\end{prop}
\begin{proof}
  $(ii)\Rightarrow (iii)$ is clear. 

  $(iii)\Rightarrow (i)$: Let $X$ be a proper smooth variety over $\mathbb{F} $. 
  Consider the action of $G_{\mathbb{F}}$ on $H^{2d}(X_{\overline{\mathbb{F} }},\mathbb{Q} _{\ell})$. 
  By the argument in the proof of Lemma \ref{key lem of Gal rep}, after replacing $\mathbb{F} $ by a finite extension, we may assume that the action factors through $G_{\mathbb{F} }^{\ell}$. 
  On the other hand, since $H^{2d}(X_{\overline{\mathbb{F} }},\mathbb{Q} _{\ell})_{G_{\mathbb{F}} }=0$, it follows that $G_{\mathbb{F}} ^{\ell}$ is nontrivial. 

  $(i)\Rightarrow (ii)$: Let $X$ be a proper smooth variety over $\mathbb{F} $.  
  By descent, we may assume that $X$ is defined over some FF $\mathbb{F}_q$.
  Then the  $G_{\mathbb{F} _q}$-representation $H^i(X_{\overline{\mathbb{F} }_q},\mathbb{Q} _{\ell})$ is pure of weight $i$.  
  Thus, by Lemma \ref{key lem of Gal rep}, for $i\neq 2r$, we have 
  $H^i(X_{\overline{\mathbb{F} }},\mathbb{Q} _{\ell}(r))^{G_{\mathbb{F} }}=0$. 
\end{proof}

\begin{thm}\label{HKFness over alg ext over FF}
    For $p\in \mathfrak{Primes}$, let $\mathbb{F} /\mathbb{F} _p$ be an algebraic extension. 
    Then the following conditions are equivalent:
    \begin{enumerate}[label=(\roman*)]
      \item For every $\ell \in \mathfrak{Primes}\backslash \{p\}$, $G_{\mathbb{F} }^{\ell}$ is nontrivial. 
      \item $\mathbb{F}$ is highly Kummer-faithful. 
      \item $\mathbb{F}$ is quasi-highly Kummer-faithful. 
      \item $\mathbb{F}$ is torally Kummer-faithful. 
    \end{enumerate}
\end{thm}
\begin{proof}
  $(i)\Leftrightarrow (ii)\Leftrightarrow (iii)$ follows from Proposition \ref{vanishing of coinvariants over alg ext over FF}. 
  Moreover, by \cite[Theorem B]{Murotani1}, we have $(i)\Leftrightarrow (iv)$. 
\end{proof}

\begin{prop}\label{vanishing of coinvarints=semi-AV-tor-finite}
  Let $k$ be a perfect field, and let $p\overset{\mathrm{def}}{=} \Char k$ and $\ell \in \mathfrak{Primes}$. 
  Then the following conditions are equivalent:
  \begin{enumerate}[label=(\roman*)]
    \item $k$ is $\ell ^{\infty}$-semi-AV-tor-finite. 
    \item For any finite extension $k_H/k$, it holds that $H^2(\mathbb{P} _{\overline{k}_H}^1,\mathbb{Q} _{\ell})_H=0$, where $H\overset{\mathrm{def}}{=}G_{k_H}$, and for any abelian variety $A$ over $k$, it holds that $H^1(A_{\overline{k}},\mathbb{Q} _{\ell})_{G_k}=0$. 
  \end{enumerate}
  In particular, if $k$ is quasi-highly Kummer-faithful, then $k$ is $\mathfrak{Primes}\backslash \{p\}^{\infty}$-semi-AV-tor-finite. 
\end{prop}
\begin{proof}
  We have $G_k$-isomorphisms
  $\mathbb{Q} _{\ell}(1)\cong H^2(\mathbb{P} _{\overline{k}} ^1,\mathbb{Q} _{\ell})^{\vee}$ and 
  $V_{\ell}(A)\cong H^1(A_{\overline{k}},\mathbb{Q} _{\ell})^{\vee}$
  for any abelian variety over $k$, where $H^2(\mathbb{P} _{\overline{k}} ^1,\mathbb{Q} _{\ell})^{\vee}$ (resp.\ $H^1(A_{\overline{k}},\mathbb{Q} _{\ell})^{\vee}$) is the dual space of $H^2(\mathbb{P} _{\overline{k}} ^1,\mathbb{Q} _{\ell})$ (resp.\ $H^1(A_{\overline{k}},\mathbb{Q} _{\ell})$).  
  Thus, the assertion follows from Propositions \ref{Ozeki-Taguchi, Prop. 2,4} and \ref{finite=T-finite+A-finite}. 
\end{proof}
\vspace{\baselineskip}
In the remainder of this section, let $k$ be a CDVF. 
Then we shall write 
\begin{itemize}
  \item $\mathcal{O} _k$ for the ring of integers of $k$, 
  \item $\underline{k}$ for the residue field of $\mathcal{O} _k$,
  \item $\overline{\underline{k}}$ for the algebraic closure of $\underline{k}$.
\end{itemize}
For a proper smooth variety $X$ over $k$, we shall write
$$\mathfrak{X} \to \Spec \mathcal{O} _k$$ 
for a proper model of $X$.
Furthermore, if proper smooth variety $X$ over $k$ has semi-stable reduction, then, for each integer $j\geq  0$, we shall write 
$$\mathfrak{X} _{\underline{k}}^{(j)}$$ 
for the disjoint union of all $(j+1)$-fold intersections of the distinct irreducible components of $\mathfrak{X} _{\underline{k}} $.
In the remainder of this section, let $p\in \mathfrak{Primes}$, and we assume that $\underline{k}/\mathbb{F} _p$ is an algebraic extension.\\

\begin{prop}\label{l-adic vanishing over CDVF}
  Let $\ell \in \mathfrak{Primes}\backslash \{p\}$. 
  Then the following conditions are equivalent:
  \begin{enumerate}[label=(\roman*)]
    \item $G_{\underline{k}}^{\ell}$ is nontrivial.
    \item For any $d$-dimensional proper smooth variety $X$ over $k$, every integer $i\neq 2d$, and every integer $r\notin \left[\max(0, i - d), \, \max(0, \min(i - 1, d - 1))\right]$, it holds that 
          $H^i(X_{\overline{k}},\mathbb{Q} _{\ell}(r))^{G_k}=0$. 
    \item For any proper smooth variety $X$ over $k$ and every integer $i\neq 0$, it holds that $H^i(X_{\overline{k}},\mathbb{Q} _{\ell})_{G_k}=0$. 
    \item For any abelian variety $A$ over $k$, it holds that $V_{\ell}(A)^{G_k}=0$. 
    \item For any semi-abelian variety $B$ over $k$, it holds that $V_{\ell}(B)^{G_k}=0$. 
    \item $k$ is stably $\mu _{\ell ^{\infty}}$-finite. 
  \end{enumerate}
\end{prop}
\begin{proof}
  $(ii)\Rightarrow (iii)\Rightarrow (v)$, $(v)\Rightarrow (iv)$, and $(v)\Rightarrow (vi)$ are clear. 

  $(vi)\Rightarrow (i)$:
  Let $k_0$ be the maximal unramified extension of $\mathbb{F}_p(\!(t)\!)$ or $\mathbb{Q}_p$ contained in $k$.
  Then $k_0$ is also stably $\mu _{\ell^\infty}$-finite.
  Thus, the assertion follows from the same argument as in \cite[Lemma 3.2]{Ozeki}.

  $(iv)\Rightarrow (i)$: 
  Let $E$ be an elliptic curve over $k$. 
  Consider the action of $G_{\underline{k}}$ on $V_{\ell}(E)^{I_k}$. 
  By the argument in the proof of Lemma \ref{key lem of Gal rep}, after replacing $\underline{k}$ by a finite extension, we may assume that the action factors through $G_{\underline{k}}^{\ell}$. 
  On the other hand, since $V_{\ell}(E)^{G_k}=0$, it follows that $G_{\underline{k}}^{\ell}$ is nontrivial.

  $(i)\Rightarrow (ii)$: 
  Let $X$ be a $d$-dimensional proper smooth variety over $k$. 
  We may assume that $X$ has projective strictly semi-stable reduction.  
  Then we consider the following $\ell$-adic weight spectral sequence:
  \[
  E_1^{a,b}
  = \bigoplus_{r \ge \max(0, -a)}
  H^{b-2r} ( \mathfrak{X} _{\overline{\underline{k}}}^{(a+2r)}, \mathbb{Q}_\ell(-r) )
  \Longrightarrow
  H^{a+b}( X_{\overline{k}}, \mathbb{Q}_\ell ) 
 \]
 Note that this spectral sequence degenerates at the $E_2$-page.
 By Proposition \ref{isom of N^d}, we have 
 $
    H^d(X_{\overline{k}},\mathbb{Q} _\ell)^{I_k} 
    \subseteq  E_2^{-d+1,2d-1}\oplus E_2^{-d+2,2d-2}\oplus \cdots \oplus E_2^{d,0}                                                 
 $ . 
  By descent, $E_2^{a,b}$ is equipped with an action of the absolute Galois group of some FF, and is pure of weight $b$.
  Hence, by Lemma \ref{key lem of Gal rep}, for $b\neq 2r$, we have $E_2^{a,b}(r)^{G_{\underline{k}} }=0$. 
  By the above inclusion, for $r\notin [0,\, d-1]$, we have $H^d(X_{\overline{k}},\mathbb{Q} _{\ell}(r))^{G_k}=0$. 
  Thus, by Lemma \ref{key lem of vanishing}, for $i\neq 2d$ and $r\notin \left[\max(0, i - d), \, \max(0, \min(i - 1, d - 1))\right]$, we have
          $H^i(X_{\overline{k}},\mathbb{Q} _{\ell}(r))^{G_k}=0$
\end{proof}

\begin{thm}
  Let $k$ be a PCDVF. 
  Then the following conditions are equivalent:
  \begin{enumerate}[label=(\roman*)]
    \item For every $\ell \in \mathfrak{Primes}\backslash \{p\}$, $G_{\underline{k}}^{\ell}$ is nontrivial. 
    \item $k^{\textnormal{perf}}$ is quasi-highly Kummer-faithful. 
  \end{enumerate}
\end{thm}
\begin{proof}
  The field $k$ satisfies the condition $(\dagger)_{d,i,r}$ if and only if $k^{\text{perf}}$ satisfies the condition $(\dagger)_{d,i,r}$. 
  Thus, the assertion follows from Proposition \ref{l-adic vanishing over CDVF}. 
\end{proof}
\vspace{\baselineskip}
In the remainder of this section, for an algebraic extension $\mathbb{F} /\mathbb{F} _p$, we shall write $$K_0(\mathbb{F} )\overset{\mathrm{def}}{=}W(\mathbb{F} )[1/p]$$ for the fraction field of the Witt ring  of $\mathbb{F} $.\\

\begin{lem}\label{key lem of p-adic case}
  Let $D$ be a $\varphi$-module over $K_0(\mathbb{F} _q)$ that is pure of nonzero weight. 
  Let $\mathbb{F} /\mathbb{F} _q$ be an algebraic extension, and assume that $G_{\mathbb{F} }^p$ is nontrivial where $p$ is the characteristic of $\mathbb{F} _q$. 
  Then we have
  \[ \left(D\otimes _{K_0(\mathbb{F} _q)}K_0(\mathbb{F} )\right)^{\varphi =1}=0.\] 
\end{lem}
\begin{proof}
  Set $V\overset{\mathrm{def}}{=}\left(D\otimes _{K_0(\mathbb{F} _q)}K_0(\overline{\mathbb{F}} _q)\right)^{\varphi =1}$. 
  Then $V$ is a $G_{\mathbb{F} _q}$-representation over $\mathbb{Q} _p$. 
  For any finite extension $\mathbb{F} _{q^m}/\mathbb{F} _q$, we have 
  \[
  V^{G_{\mathbb{F} _{q^m}}}=\left(\left(D\otimes _{K_0(\mathbb{F} _q)}K_0(\overline{\mathbb{F}} _q)\right)^{G_{\mathbb{F} _{q^m}}}\right)^{\varphi =1}=\left(D\otimes _{K_0(\mathbb{F} _q)}K_0(\mathbb{F} _{q^m})\right)^{\varphi =1}=0.
  \] 
  Thus, by Lemma \ref{key lem of Gal rep}, we have 
  $$\left(D\otimes _{K_0(\mathbb{F} _q)}K_0({\mathbb{F} })\right)^{\varphi =1}=V^{G_\mathbb{F} }=0.$$ 
\end{proof}

\begin{prop}\label{p-adic vanishing over MCDVF}
  Let $k$ be an MCDVF. 
  Then the following conditions are equivalent:
  \begin{enumerate}[label=(\roman*)]
    \item $G_{\underline{k}}^p$ is nontrivial. 
    \item For any $d$-dimensional proper smooth variety $X$ over $k$, every integer $i\neq 2d$, and every integer $r\notin \left[\max(0, i - d), \, \max(0, \min(i - 1, d - 1))\right]$, it holds that 
          $H^i(X_{\overline{k}},\mathbb{Q} _{p}(r))^{G_k}=0$. 
    \item For any proper smooth variety $X$ over $k$ and an integer $i\neq 0$, it holds that $H^i(X_{\overline{k}},\mathbb{Q} _{p})_{G_k}=0$. 
    \item $k$ is $p^{\infty}$-semi-AV-tor-finite. 
  \end{enumerate}
\end{prop}
\begin{proof}
  $(ii)\Rightarrow (iii)$ is clear. 
  $(iii)\Rightarrow (iv)$ follows from Proposition \ref{vanishing of coinvarints=semi-AV-tor-finite}. 
  
  $(iv)\Rightarrow (i)$: Let $k_0$ be the maximal unramified extension of $\mathbb{Q} _p$ contained in $k$. 
  Then $k_0$ is also $p^{\infty}$-semi-AV-tor-finite. 
  Thus, the assertion follows from \cite[Lemma 3.2]{Ozeki}.

  $(i)\Rightarrow (ii)$: Let $X$ be a $d$-dimensional proper smooth variety over $k$. 
  We may assume that $X$ has projective strictly semi-stable reduction. 
  Note that $\mathfrak{X}_{\underline{k}}$ has the structure of a projective SNCL-variety.
  Let $\mathfrak{X}_0$ be a descent of $\mathfrak{X}_{\underline{k}}$ to some FF $\mathbb{F}_q$ (cf.\ \cite[Lemma 2.2]{Nakayama}).
  Consider the following two $p$-adic weight spectral sequences:
  $$E_1^{a,b}(\mathfrak{X} _{\underline{k}})= \bigoplus_{r \ge \max(0, -a)} H_{\textnormal{crys}}^{b-2r} ( \mathfrak{X} _{\underline{k}} ^{(a+2r)}/K_0(\underline{k}))(-r)\Longrightarrow H_{\textnormal{log-crys}}^{a+b}( \mathfrak{X} _{\underline{k}} /K_0(\underline{k})),$$
  $$E_1^{a,b}(\mathfrak{X} _{0})= \bigoplus_{r \ge \max(0, -a)} H_{\textnormal{crys}}^{b-2r} ( \mathfrak{X} _{0} ^{(a+2r)}/K_0(\mathbb{F} _q))(-r)\Longrightarrow H_{\textnormal{log-crys}}^{a+b}( \mathfrak{X} _{0} /K_0(\mathbb{F} _q)).$$
  Note that these spectral sequences degenerate at the $E_2$-page. 
  By the base change theorem, there is a canonical isomorphism
  $H_{\text{crys}}^i(\mathfrak{X} _0 ^{(j)}/K_0(\mathbb{F} _q))\otimes _{K_0(\mathbb{F} _q)}K_0(\underline{k})\to H_{\text{crys}}^i(\mathfrak{X} _{\underline{k}} ^{(j)}/K_0(\underline{k}))$.
  Hence, we have $\varphi$-isomorphism $E_2^{a,b}(\mathfrak{X} _{0})\otimes _{K_0(\mathbb{F} _q)} K_0(\underline{k})\cong E_2^{a,b}(\mathfrak{X} _{\underline{k}})$. 
  Now, as a $\varphi$-module, $E_2^{a,b}(\mathfrak{X} _{0})$ is pure of weight $b$.
  Thus, by Lemma \ref{key lem of p-adic case}, for $b\neq 2r$, we have $E_2^{a,b}(\mathfrak{X} _{\underline{k}})(r)^{\varphi =1}=0$. 
  By the $C_{\text{st}}$-conjecture and Proposition \ref{p-adic isom of N^d}, we have 
  \[
  \begin{aligned}
    H^d(X_{\overline{k}},\mathbb{Q} _p)^{G_k} &\cong \Hom _{G_k}\left(\mathbb{Q} _p, H^d(X_{\overline{k}},\mathbb{Q} _p)\right)\\
                                              &\cong \Hom _{\varphi, N, \mathrm{Fil}^{\bullet}} \left(K_0(\underline{k}), D_{\textnormal{st}}(H^d(X_{\overline{k}},\mathbb{Q} _p))\right)\\
                                              &\cong \Hom _{\varphi, N, \mathrm{Fil}^{\bullet}} \left(K_0(\underline{k}), H_{\textnormal{log-crys}}^d( \mathfrak{X} _{\underline{k}} /K_0)\right)\\
                                              &\subseteq  \Ker \left( H_{\textnormal{log-crys}}^d( \mathfrak{X} _{\underline{k}} /K_0(\underline{k})) \xrightarrow{N} H_{\textnormal{log-crys}}^d( \mathfrak{X} _{\underline{k}} /K_0(\underline{k})) \right)^{\varphi =1}\\  
                                              &\subseteq  E_2^{-d+1,2d-1}(\mathfrak{X} _{\underline{k}})^{\varphi =1}\oplus E_2^{-d+2,2d-2}(\mathfrak{X} _{\underline{k}})^{\varphi =1}\oplus \cdots \oplus E_2 ^{d,0}(\mathfrak{X} _{\underline{k}})^{\varphi =1}.\\                                             
  \end{aligned}
  \]
  --- where the second-to-last inclusion is obtained by evaluating each morphism at $1\in K_0(\underline{k})$. 
  Hence, for $r\notin [0,\, d-1]$, we have $H^d(X_{\overline{k}},\mathbb{Q} _p(r))^{G_k}=0$. 
  Thus, by Lemma \ref{key lem of vanishing}, for $i\neq 2d$ and $r\notin \left[\max(0, i - d), \, \max(0, \min(i - 1, d - 1))\right]$, we have $H^i(X_{\overline{k}},\mathbb{Q} _{p}(r))^{G_k}=0$. 
\end{proof}

\begin{thm}\label{QHKFness over MCDVF}
  Let $k$ be an MCDVF.
  Then the following conditions are equivalent:
  \begin{enumerate}[label=(\roman*)]
    \item $\underline{k}$ is quasi-finite. 
    \item $k$ is quasi-highly Kummer-faithful. 
    \item $k$ is Kummer-faithful. 
    \item $k$ is $\mathfrak{Primes}^{\infty}$-semi-AV-tor-finite. 
  \end{enumerate}
\end{thm}
\begin{proof}
  $(iii)\Rightarrow (iv)$ is clear. 
  $(i)\Leftrightarrow (ii)\Leftrightarrow (iv)$ follows from Propositions \ref{l-adic vanishing over CDVF} and \ref{p-adic vanishing over MCDVF}. 
  $(i)\Rightarrow (iii)$ follows from \cite[Theorem B]{Murotani1} and \cite[Proposition 3.7]{Murotani2}. 
\end{proof}

\begin{cor}\label{QHKFness of unram etx over MLF}
  Let $K$ be an MLF, $F/K$ an unramified extension. 
  Then the following conditions are equivalent:
  \begin{enumerate}[label=(\roman*)]
    \item $F/K$ is a quasi-finite extension. 
    \item $F$ is quasi-highly Kummer-faithful. 
    \item $F$ is Kummer-faithful. 
    \item $F$ is $\mathfrak{Primes}^{\infty}$-semi-AV-tor-finite.  
  \end{enumerate}
\end{cor}
\begin{proof}
  $(iii)\Rightarrow (iv)$ is clear. 
  By Theorem \ref{QHKFness over MCDVF}, if $F/K$ is a quasi-finite extension, then the completion of $F$ is quasi-highly Kummer-faithful. Since any perfect subfield of quasi-highly Kummer-faithful field is quasi-highly Kummer-faithful, $(i)\Rightarrow(ii)$ holds.
  $(ii)\Rightarrow (iii)$ follows from \cite[Proposition 2.8]{Ozeki-Taguchi}. 
  Assume that a perfect field $F$ is $\mathfrak{Primes}^{\infty}$-semi-AV-tor-finite. 
  Then, for every $\ell \in \mathfrak{Primes}\backslash \{p\}$, $F$ is stably $\mu _{\ell ^{\infty}}$-finite. 
  Thus, by \cite[Lemma 3.2]{Ozeki}, we have $(iv)\Rightarrow (i)$. 
\end{proof}

\vspace{\baselineskip}
\bibliographystyle{amsplain}
\bibliography{Weights_in_etale_cohomology_over_MLFs}

\end{document}